\documentclass[10pt, a4paper]{article}

\usepackage{amssymb, amsmath, textcomp, amsthm}
\usepackage{bbm,dsfont}
\numberwithin{equation}{section}

\title{ Sharp bounds for the cubic Parsell--Vinogradov system in two dimensions}
\author{Jean Bourgain, Ciprian Demeter, Shaoming Guo}
\date{}

\def\R{\mathbb{R}}
\def\N{\mathbb{N}}
\def\C{\mathbb{C}}\def\nint{\mathop{\diagup\kern-13.0pt\int}}

\def\lesim{\lesssim}

\def\beq{\begin{equation}}
\def\endeq{\end{equation}}
\def\bg{\begin{gathered}}
\def\eg{\end{gathered}}

\def\P{\Phi}
\def\w{w_{B_N}}

\def\83{\frac{8}{3}}
\def\38{\frac{3}{8}}

\def\mc{\mathcal}

\theoremstyle{plain}
\newtheorem{thm}{Theorem}[section]
\newtheorem{prop}[thm]{Proposition}

\newtheorem{lem}[thm]{Lemma}
\newtheorem{cor}[thm]{Corollary}
\newtheorem{defi}[thm]{Definition}
\newtheorem{claim}[thm]{Claim}

\newtheorem*{conj*}{Conjecture}
\newtheorem{conj}[thm]{Conjecture}
\newtheorem*{openproblem*}{Open Problem}

\begin{document}


\maketitle

\begin{abstract}
We prove a sharp decoupling for a certain two dimensional surface in $\R^9$. As an application, we obtain the full range of expected estimates for the cubic Parsell--Vinogradov system in two dimensions.
\end{abstract}

\smallskip
%
%
\let\thefootnote\relax\footnote{
The first author is partially supported by the NSF grant DMS-1301619. The second  author is partially supported  by the NSF Grant DMS-1161752\\
AMS subject classification: Primary 11L07; Secondary 42A45}

\section{Introduction and statements of main results}

For $d, s\ge 1$ and $k\ge 2$, consider the integer solutions
\beq\label{2305f1.5}
x_{1, j}, x_{2, j}, ..., x_{d, j}, y_{1, j}, y_{2, j}, ..., y_{d, j},\;1\le j\le s
\endeq
of the system of Diophantine equations (often referred to as the {\em Parsell--Vinogradov} system)
\beq\label{2305e1.5}
\sum_{j=1}^s x_{1, j}^{i_1} x_{2, j}^{i_2}... x_{d, j}^{i_d}=\sum_{j=1}^s y_{1 ,j}^{i_1} y_{2, j}^{i_2}... y_{d, j}^{i_d}.
\endeq
Here $0\le i_1, i_2, ..., i_d\le k$ are all integers such that $1\le i_1+i_2+...+i_d\le k$. For instance, when $d=1$, the system \eqref{2305e1.5} consists of one equation
\beq
\sum_{j=1}^s x_{j}^{i}=\sum_{j=1}^s y_{j}^{i}, \text{ with } 1\le i\le k.
\endeq
When $d=k=2$, the system \eqref{2305e1.5} becomes
\beq
\begin{split}
& x_{1, 1}+x_{1, 2}+...+x_{1, s}=y_{1, 1}+y_{1, 2}+...+y_{1, s},\\
& x_{2, 1}+x_{2, 2}+...+x_{2, s}=y_{2, 1}+y_{2, 2}+...+y_{2, s},\\
& x^2_{1, 1}+x^2_{1, 2}+...+x^2_{1, s}=y^2_{1, 1}+y^2_{1, 2}+...+y^2_{1, s},\\
& x^2_{2, 1}+x^2_{2, 2}+...+x^2_{2, s}=y^2_{2, 1}+y^2_{2, 2}+...+y^2_{2, s},\\
& x_{1, 1}x_{2, 1}+x_{1, 2} x_{2, 2}+...+x_{1, s}x_{2, s}=y_{1, 1}y_{2, 1}+y_{1, 2}y_{2, 2}+...+y_{1, s} y_{2, s}.
\end{split}
\endeq
Define
$$
\mc{K}_{j, k}=\frac{jk}{j+1} \binom{k+j}{j}.
$$
For large  $N$, we let $J_{s, d, k}(N)$ denote the number of integer solutions \eqref{2305f1.5} of the system of equations \eqref{2305e1.5} satisfying $1\le x_{1, j}, ..., x_{d, s}, y_{1, j}, ..., y_{d, j}\le N$ for each $1\le j\le s$. Parsell, Prendiville and Wooley provided a lower bound for $J_{s, d, k}(N)$.
\begin{thm}(\cite{Par})
We have the lower bound on the number of integer solutions of \eqref{2305e1.5}
\beq\label{2305e1.12}
J_{s, d, k}(N)\gtrsim N^{sd}+\sum_{j=1}^d N^{(2s-1)j+d-\mc{K}_{j, k}}.
\endeq
Here the implicit constant depends only on $s, d$ and $k$.
\end{thm}
The right hand side of \eqref{2305e1.12} has $d+1$ terms, which indicates that there might be  about $d$ many regimes when estimating $J_{s, d, k}(N)$. When $d=1$ and $k\ge 2$, \eqref{2305e1.12} becomes
$$
J_{s, 1, k}(N)\gtrsim N^s+N^{2s-\mc{K}_{1, k}}=N^s+N^{2s-\frac{k(k+1)}{2}}.
$$
From this, we see that we have the following two regimes
$$
2s\ge k(k+1) \text{ and } 2s\le k(k+1).
$$
When $d=2$ and $k\ge 2$, we have
$$
J_{s, 2, k}(N)\gtrsim N^{2s}+N^{2s+1-\frac{k(k+1)}{2}}+N^{4s-\frac{k(k+2)(k+1)}{3}}\sim N^{2s} +N^{4s-\frac{k(k+2)(k+1)}{3}}.
$$
We see that there are still only two regimes in this case
$$
2s\ge \frac{k(k+1)(k+2)}{3} \text{ and } 2s\le \frac{k(k+1)(k+2)}{3}.
$$
We analyze one more case. When $d=3$, the right hand side of \eqref{2305e1.12} becomes
\beq\label{2405e1.11}
J_{s, 3, k}(N)\gtrsim N^{3s}+N^{2s+2-\frac{k(k+1)}{2}}+N^{4s+1-\frac{k(k+1)(k+2)}{3}}+N^{6s-\frac{k(k+1)(k+2)(k+3)}{8}}.
\endeq
It turns out that when $2\le k\le 4$, the last expression is equivalent to
$$
J_{s, 3, k}(N)\gtrsim N^{3s}+N^{6s-\frac{k(k+1)(k+2)(k+3)}{8}}.
$$
When $k\ge 5$, a new regime appears
$$
J_{s, 3, k}(N)\gtrsim N^{3s}+N^{4s+1-\frac{k(k+1)(k+2)}{3}}+N^{6s-\frac{k(k+1)(k+2)(k+3)}{8}}.
$$
That means that we have three different behaviours of $J_{s, 3, k}$, depending on which of the following intervals $2s$ belongs to
$$
[2, \frac{2k(k+1)(k+2)}{3}-2],\;\;\;\; (\frac{2k(k+1)(k+2)}{3}-2,\frac{k(k+1)(k+2)(3k+1)}{24}-1]$$or $$(\frac{k(k+1)(k+2)(3k+1)}{24}-1, \infty).
$$
This discussion already gives an indication about the growth in complexity of the quantities $J_{s, d, k}(N)$ as $d,k$ grow large.
\bigskip

Closely related to the number of solutions \eqref{2305f1.5} of the system of equations \eqref{2305e1.5} are several sharp decoupling inequalities. For $d\ge 1$ and $k\ge 2$, let $\mc{S}_{d, k}$ be the $d$ dimensional surface in $\R^{n}$ with
\beq\label{2405e1.15}
n=\binom{d+k}{k}-1,
\endeq
defined by
$$
\mc{S}_{d, k}=\{\Phi(t_1, t_2, ..., t_d): (t_1, t_2, ..., t_d)\in [0, 1]^d\}.
$$
 Here the entries of $\Phi=\Phi_{d,k}:\R^d\to \R^n$ consist of all the monomials $t_{1}^{i_1} t_{2}^{i_2}... t_{d}^{i_d}$ with $1\le i_1+i_2+...+i_d\le k$.  For instance, when $d=1$, we have $n=k$ and
$$
\mc{S}_{1, k}=\{(t, t^2, ..., t^k): t\in [0, 1]\}.
$$
When $d=2$ and $k=2$, we have $n=5$ and
$$
\mc{S}_{2, 2}=\{(t_1, t_2, t_1^2, t_1 t_2, t_2^2): (t_1, t_2)\in [0, 1]^2\}.
$$

\medskip

The lower bounds in \eqref{2305e1.12} were  also conjectured in \cite{Par} to be upper bounds, up to a factor $N^{\epsilon}$, for  arbitrarily small $\epsilon>0$. There has been significant progress towards this conjecture in recent years. Two types of methods have been employed. The first one is number theoretical and revolves around the efficient congruencing method of Wooley. The reader may consult \cite{Woo}, \cite{Par} and other subsequent papers of Wooley and his collaborators for the results using these methods.
\medskip

More recently, the first and second authors have developed in \cite{BD1} an alternative approach called {\em decouplings}, that employs solely Fourier analytic techniques. We next recall the relevant machinery.
\medskip

For a measurable subset $R\subset [0, 1]^d$ and a measurable function $g:R\to\C$,  define the so-called {\em extension} operator associated to $\mc{S}_{d, k}$ (restricted to $R$) by
$$
E^{(d, k)}_R g(x)=\int_{R} g(t)e(t_1 x_1+...+t_d x_d+ t_1^2 x_{d+1}+...+t_d^k x_n)dt.
$$
Here and in the following we write $e(z)=e^{2\pi iz}$, $x=(x_1,\ldots,x_n)$ and $dt=dt_1\ldots dt_d$. Also, for a ball $B=B(c,R)$ in $\R^n$ we will use the weight
$$w_B(x)=(1+\frac{\|x-c\|}{R})^{-C}$$
where $C$ is a large enough constant whose value will not be specified.
For each $p\ge q\ge 2$ and $0<\delta<1$ we denote by $V^{(d, k)}_{p,q}(\delta)$ the smallest constant such that
\beq\label{2305e1.20}
\|E^{(d, k)}_{[0, 1]^d}g\|_{L^p(w_{B})}\le V^{(d, k)}_{p,q}(\delta) (\sum_{\substack{\Delta: \text{ cube in } [0, 1]^d\\ l(\Delta)=\delta}}\|E^{(d, k)}_{\Delta}g\|_{L^p(w_{B})}^q)^{1/q}
\endeq
holds for each ball $B\subset \R^{n}$ with radius $\delta^{-k}$ and each measurable $g:[0,1]^d\to \C$. The summation on the right is understood to be over a finitely overlapping cover of $[0,1]^d$ with cubes $\Delta$ whose side length is $l(\Delta)=\delta$.
Such an inequality will be referred to throughout the paper as an $l^qL^p$ decoupling.
\medskip

Let us state several results that are most relevant to us.
\begin{thm}(Bourgain, Guth and Demeter \cite{BDG})\label{2305thm1.2}
Take $d=1$. For each $k\ge 2$, for each $\epsilon>0$, we have
$$
V^{(1, k)}_{k(k+1),2}(\delta)\lesim_{k, \epsilon} \delta^{\epsilon}.
$$
Moreover, this implies the following sharp upper bound
$$
J_{s, 1, k}(N) \lesim_{s, k, \epsilon} N^{s+\epsilon}+N^{2s-\frac{k(k+1)}{2}+\epsilon}.
$$
\end{thm}
This completely settled the problem in dimension $d=1$. As a preparation for stating our results in dimension $d=2$, we first state a conjecture of purely linear algebra nature. We formulate it in arbitrary dimensions.

For each $t\in[0,1]^d$ and $1\le l\le k-1$, we let  (recall $\Phi=\Phi_{d,k}$ was introduced earlier)
 $\mc{M}^{(l)}(t)$ denote the $n\times \big(\binom{d+l}{l}-1 \big)$ matrix whose columns are the vectors $\Phi^{(\alpha)}(t)$, with $\alpha$ running through all the multi-indices with $1\le |\alpha|\le l$.

Take a linear space $V=\langle v_1, v_2, ..., v_{dim(V)}\rangle\subset \R^{n}$ with $n$ given by \eqref{2405e1.15}. For convenience, we let all $v_i$ be column vectors. Define
$$
\mc{M}^{(l)}_V(t)=(v_1, v_2, ..., v_{dim(V)})^T \times \mc{M}^{(l)}(t).
$$
Here ``$\times $'' refers to the product of two matrices. Hence for each $t\in [0, 1]^d$, $\mc{M}_V^{(l)}(t)$ is a $dim(V)\times \big(\binom{d+l}{l}-1 \big)$ matrix. We consider the following conjecture.

\begin{conj}\label{conjecture0}
For each $d\ge 2$ and $k\ge 2$, each $1\le l\le k-1$ and each linear subspace $V\subset \R^{n}$ with dimension $dim(V)$, the matrix $\mc{M}_V^{(l)}(t)$ has at least one minor  of order
$$
\Big[\frac{dim(V)\cdot \big(\binom{d+l}{l}-1\big)}{\binom{d+k}{k}-1}\Big]+1,
$$
whose determinant, viewed as a function of $t\in [0, 1]^d$, does not vanish identically.
\end{conj}

Now we are ready to recall a result in dimension $d=2$ due to the first and second authors.
\begin{thm}(d=2, \;\cite{BD2})\label{2405thm1.4}
Assume Conjecture \ref{conjecture0} holds for $l=1$, $d=2$ and some $k\ge 2$. Then for each $\epsilon>0$, we have the sharp upper bounds
$$
V^{(2, k)}_{p,p}(\delta)\lesim_{p, \epsilon} \delta^{2(\frac{1}{2}-\frac{1}{p})+\epsilon}, \text{ for all } p\le k(k+3)-2.
$$
Moreover, Conjecture \ref{conjecture0} is verified\footnote{This conjecture was formulated slightly differently in \cite{BD2}, but the two formulations are equivalent} for $k=2$ and also for $k=3$. This further implies the sharp upper bounds
$$
J_{s, 2, 2}(N)\lesim N^{2s+\epsilon} + N^{4s-8} \text{ for all } s\ge 1,
$$
and
$$
J_{s, 2, 3}(N)\lesim N^{2s+\epsilon} + N^{4s-20} \text{ for all } 1\le s\le 8.
$$
\end{thm}
In light of the previous discussion, this completely settled Parsell--Prendiville--Wooley conjecture in the case $d=k=2$.

We note that while the above estimate for $J_{s, 2, 3}(N)$ is sharp in the range $1\le s\le 8$, it falls short of recovering the full expected range $s\ge 1$, due to the gap between $8$ and the critical index $s=10$. The main new result proved here  bridges this gap, thus solving the Parsell--Prendiville--Wooley conjecture in the cubic case $d=2, k=3$.

\begin{thm}\label{2405thm1.5}($d=2,k=3$)
Conjecture \ref{conjecture0} holds for $d=2,k=3$ and $1\le l\le 2$. We also have the sharp bound
\beq
\label{abc13}
V^{(2,3)}_{20,20}(\delta)\lesim_{\epsilon, p} \Big(\frac{1}{\delta}\Big)^{2(\frac12-\frac1{20})+\epsilon}.
\endeq
This further implies the sharp upper bound, in the full expected range
\beq
\label{abc52}
J_{s, 2, 3}(N)\lesim_\epsilon N^{2s+\epsilon} + N^{4s-20+\epsilon} \text{ for all } s\ge 1.
\endeq
\end{thm}
The fact that \eqref{abc52} follows from \eqref{abc13} is rather standard, see for example section 2 in \cite{BD2}.
\medskip

The proof of Conjecture \ref{conjecture0} for $(d,k,l)=(2,3,2)$ is quite computationally involved, so we postpone it to the Appendix. Inequality \eqref{abc13} will follow by refining the decoupling approach developed over the last three years. Two papers, namely \cite{BD2} and \cite{BDG} will play a key inspirational role. The main novelties can be described as follows.

One of the cornerstones of our approach here (and in \cite{BDG}) is a certain ball inflation lemma, which is some sort of  multilinear Kakeya inequality in disguise. Here, this result is Lemma \ref{main1}. When $l=1$, this Lemma requires an $l^q$ summation with $q<\frac{40}{9}$ (see the remarks after the lemma). This in turn forces us to revisit the decoupling theory from \cite{BD2} for the quadratic surface $\mc{S}_{2,2}$. This is explained in Section \ref{abcsec4}. More precisely, the sharp $l^8L^8$ theory for $\mc{S}_{2,2}$ was established in \cite{BD2}, but for the reason described above we now need to understand the sharp $l^qL^{8}$ decoupling for $q<\frac{40}{9}$. One subtle aspect of this new endeavor is that this decoupling is only playing in our favor if we also have $q\ge \frac83$.
The key new element in establishing the $l^qL^8$ ($\frac83\le q<8$) decoupling for $\mc{S}_{2,2}$ versus the $l^{8}L^8$ decoupling from \cite{BD2} is in the way the lower dimensional term from the Bourgain--Guth scheme is estimated. While in \cite{BD2} this is estimated using a trivial decoupling, in the current situation this cheap approach is no longer sufficient. Instead, we need to invoke the $l^2$ decoupling for the parabola from \cite{BD1}, combined with certain dimension-reduction lemmas. These lemmas  of independent interest are proved in Section \ref{abcsec2}.

We decided to run the iteration argument in Section \ref{abcsec5} with $q=\frac83$, but, because of the reasons described above we could have used any $q\in[\frac83,\frac{40}9]$. We recall that in \cite{BDG} the value $q=2$ was used. The use of the new value $q=\frac83$ will force a whole host of subtle differences in Section \ref{abcsec5}, compared the argument in \cite{BDG}. But the fact that we eventually care about values of $p$ very close to $20>\frac83$ will always play in our favor.
\medskip

We believe that Conjecture \ref{conjecture0} should be true for all values of $k,d,l$ mentioned there, but at the moment proving this seems a real challenge. Our proof from the Appendix for the special case $(d,k,l)=(2,3,2)$ may give an indication on the level of complexity of the general case. Apart from Conjecture \ref{conjecture0} and the numerology which gets increasingly more complicated as $d,k$ get larger, we believe that there should not be any new serious obstacle for settling the full Parsell--Prendiville--Wooley conjecture. In particular, we believe that the analytic part of the machinery needed for the general case of $d,k$ is already present in this paper.
\medskip

Unlike \cite{BDG} where we opted for a high level of details, the presentation here is a bit less detail oriented when it comes about technicalities. The reader interested in all details is referred to both  \cite{BDG} and also to the study guide \cite{BD4}.

\section{Notation}
\bigskip

Throughout the paper we will write $A\lesssim_{\upsilon}B$ to denote the fact that $A\le CB$ for a certain implicit constant $C$ that depends on the parameter $\upsilon$. Typically, this parameter is either $\epsilon$ or $K$. The implicit constant will never depend on the scale $\delta$, on the balls we integrate over, or on the function $g$. It will however most of the times depend on  $n,k,d$ and on the Lebesgue index $p$. Since these can  be thought of as being fixed parameters, we will in general not write $\lesssim_{p,n,k,d}$.

We will denote by $B_R$ an arbitrary ball of radius $R$. We use the following two notations for averaged integrals
$$\nint_B F=\frac1{|B|}\int_BF,$$
$$\|F\|_{L^p_\sharp(w_B)}=(\frac1{|B|}\int|F|^pw_B)^{1/p}.$$
$|A|$ will refer to either the cardinality of $A$ if $A$ is finite, or to its Lebesgue measure if $A$ has positive measure.

\section{Dimensional reductions}
\label{abcsec2}
In this section we present a few auxiliary results that will be used a few times throughout the paper. They are also expected to play a role in future investigations.

For the rest of this section we will assume $p\ge q\ge 1$. Given a manifold $$\{(v,Q(v)):v\in\R^{d}\}$$
associated with $Q:\R^{d}\to \R^{d'}$, its extension operator will be defined as follows
$$E_Vg(x,x')=\int_Vg(v)e(xv+x'Q(v))dv.$$
Here $V$ is an arbitrary measurable set in $\R^d$, $g$ is an arbitrary complex valued measurable function on $\R^d$ and $(x,x')\in\R^{d}\times \R^{d'}.$

The first lemma shows how to reduce the dimension of the ambient space for the manifold.

\begin{lem}
\label{abc8}
Let $Q_i:\R^{d_0}\to\R^{d_i}$, $i=1,2$ be measurable. Fix  $U_1,\ldots,U_l$, an arbitrary measurable partition of $[0,1]^{d_0}$ and fix $B$,  an arbitrary measurable subset of $\R^{d_0+d_1}$. For $i=1,2$, let $E^{(i)}$ denote the extension operators associated with the manifolds $\mc{M}_i$ defined as follows
$$\mc{M}_1=\{(u,Q_1(u)):u\in\R^{d_0}\},$$
$$\mc{M}_2=\{(u,Q_1(u),Q_2(u)):u\in\R^{d_0}\}.$$
Let $C$ be a number such that the inequality
$$\|E^{(1)}_{[0,1]^{d_0}}g\|_{L^p(B)}\le C(\sum_i\|E_{U_i}^{(1)}g\|_{L^p(B)}^q)^{1/q}$$
holds for all measurable $g$.

Then for each measurable set $B'\subset \R^{d_2}$ and for each measurable $h$ we have
$$\|E^{(2)}_{[0,1]^{d_0}}h\|_{L^p(B\times B')}\le C(\sum_i\|E_{U_i}^{(2)}h\|_{L^p(B\times B')}^q)^{1/q}.$$
\end{lem}
\begin{proof}
The argument is nothing else but a combination of Fubini and Minkowski's inequality for integrals. We include the argument for readers' convenience.

Fix $h$ and $B'$. We will use $x_i$ to denote elements of $\R^{d_i}$. For $x_2\in \R^{d_2}$ define
$$g_{x_2}:\R^{d_0}\to\R,\;\;\;g_{x_2}(u)=h(u)e(x_2Q_2(u)).$$
Note that for each measurable set $U\subset \R^{d_0}$ we have
\begin{equation}
\label{abc7}
E_U^{(2)}h(x_0,x_1,x_2)=E_U^{(1)}g_{x_2}(x_0,x_1).
\end{equation}
Thus
$$\|E^{(2)}_{[0,1]^{d_0}}h\|_{L^p(B\times B')}^p=\int_{B'}\|E_{[0,1]^{d_0}}^{(1)}g_{x_2}\|_{L^p_{x_0,x_1}(B)}^pdx_2$$
$$\le C^p\int_{B'}(\sum_i\|E_{U_i}^{(1)}g_{x_2}\|_{L^p_{x_0,x_1}(B)}^q)^{p/q}dx_2.$$
The last expression can be dominated using Minkowski's inequality for $L^{p/q}_{x_2}$ (recall $p\ge q$) by
$$\le C^p(\sum_i\|E_{U_i}^{(1)}g_{x_2}\|_{L^p_{x_0,x_1,x_2}(B\times B')}^q)^{p/q}.$$
Another application of \eqref{abc7} will close the argument.
\end{proof}

The next lemma shows how to reduce the dimension of the manifold.

\begin{lem}
\label{abc9}
Let $Q_1:\R\to\R^{d_1}$ be measurable. Fix  $I_1,\ldots,I_l$, an arbitrary measurable partition of $[0,1]$ and fix $B$,  an arbitrary measurable subset of $\R^{1+d_1}$. For $i=1,3$, let $E^{(i)}$ denote the extension operator associated with the manifolds $\mc{M}_i$ defined as follows
$$\mc{M}_1=\{(r,s,Q_1(r)):r\in\R\ , s\in \R\},$$
$$\mc{M}_3=\{(r,Q_1(r)):r\in\R\}.$$
Let $C$ be a number such that the inequality
$$\|E^{(3)}_{[0,1]}h\|_{L^p(B)}\le C(\sum_i\|E_{I_i}^{(3)}h\|_{L^p(B)}^q)^{1/q}$$
holds for all measurable $h$.

Then for each measurable set $B'\subset \R$ and for each measurable $g$ we have
$$\|E^{(1)}_{[0,1]^{2}}g\|_{L^p(B\times B')}\le C(\sum_i\|E_{I_i\times [0,1]}^{(1)}g\|_{L^p(B\times B')}^q)^{1/q},$$
where $B\times B'$ is the subset of $\R^{2+d_1}$ defined (atypically)  as
$$\{(x_1, x_2 , x_3):\;(x_1, x_3)\in B,\; x_2\in B'\}.$$
\end{lem}
\begin{proof}
We note that for each measurable $I\subset \R$ and each $x_1,x_2\in\R$, $x_3\in \R^{d_1}$ we have
$$E^{(1)}_{I\times [0,1]}g(x_1,x_2,x_3)=E_I^{(3)}h_{x_2}(x_1,x_3),$$
where
$$h_{x_2}(r)=\int_{[0,1]}g(r,s)e(sx_2)ds.$$
The rest of the argument follows exactly like the one for Lemma \ref{abc9}. The details are left to the reader.
\end{proof}

Combining the two lemmas we get the following consequence.
\begin{cor}
\label{abc9'}
Let $Q_1:\R\to\R^{d_1}$ and $Q_2:\R^2\to \R^{d_2}$ be measurable. Let as before $E^{(i)}$, $i=2,3$, denote the extension operators associated with the manifolds
$$\mc{M}_2=\{(r,s,Q_1(r),Q_2(r,s)):r,s\in\R\}.$$
and
$$\mc{M}_3=\{(r,Q_1(r)):r\in\R\}.$$Fix  $I_1,\ldots,I_l$, an arbitrary measurable partition of $[0,1]$ and fix $B$,  an arbitrary measurable subset of $\R^{1+d_1}$. Let $C$ be a number such that the inequality
$$\|E^{(3)}_{[0,1]}h\|_{L^p(B)}\le C(\sum_i\|E_{I_i}^{(3)}h\|_{L^p(B)}^q)^{1/q}$$
holds for all measurable $h:[0,1]\to\C$. Then for each measurable set $B'\subset \R^{1+d_2}$ and for each measurable $h:[0,1]^2\to\C$ we have
$$\|E^{(2)}_{[0,1]^2}h\|_{L^p(B\times B')}\le C(\sum_i\|E_{I_i\times [0,1]}^{(2)}h\|_{L^p(B\times B')}^q)^{1/q},$$
where here $B\times B'$ is the subset of $\R^{2+d_1+d_2}$ defined (atypically)  as
$$\{(x_1, x_2 , x_3, x_4):\;(x_1, x_3)\in B,\; (x_2,x_4)\in B'\}.$$
\end{cor}

We will apply this corollary with $\mc{M}_2=\mc{S}_{2,2}$ and $Q_1(r)=r^2$, see the proof of Claim \ref{2903claim3.3}. Also, we will apply Lemma \ref{abc8}  with $\mc{M}_1=\mc{S}_{2,2}$ and $\mc{M}_2=\mc{S}_{2,3}$, in order to derive inequality \eqref{0720e3.35h}. In each case $I_i$ will be intervals and $B,B'$ will be balls of the same radius. The relation between the size of the intervals and the radii of the balls will be different in the two cases. We must also mention that we will in fact use weighted versions of these results, whose proofs are only technical variations of the proofs given above.

\section{Parabolic rescaling}
We will repeatedly use the following result (see Proposition 7.1 from \cite{BD2}), that will be referred to as  parabolic rescaling.
\begin{lem}
\label{abc18}
Let $k\ge 2$, and let
$n=\binom{k+2}{k}-1$, corresponding to this $k$ and $d=2$ as in \eqref{2405e1.15}.
Let also $0<\delta<\sigma\le 1$.

Then for each square $R\subset [0, 1]^2$ with side length $\sigma$ and each ball $B\subset \R^n$ with radius $\delta^{-k}$ we have
\beq
\|E_{R}^{(2,k)} g\|_{L^p(w_B)} \le V_{p,q}^{(2,k)}(\frac{\delta}{\sigma}) (\sum_{R'\subset R:\; l(R')=\delta} \|E_{R'}^{(2,k)} g\|_{L^{p}(w_B)}^{q})^{1/q}.
\endeq
\end{lem}


\section{A new decoupling for $\mc{S}_{2,2}$}\label{0720section22}
\label{abcsec4}

Recall the following two dimensional surface in $\R^5$ introduced earlier
\beq\label{0604e1.1}
\mc{S}_{2,2}=\{(r, s, r^2, rs, s^2): (r, s)\in [0, 1]^2\}.
\endeq
Throughout this section we will simplify notation and write $\mc{S}$ for $\mc{S}_{2,2}$ and $E$ for $E^{(2,2)}$.

For $p,q\ge 2$ and $N\in[1,\infty)$ we denote by $V(N, p, q)$ the smallest constant such that
\beq\label{3005e1.2}
\|E_{[0, 1]^2}g\|_{L^p(w_{B_N})}\le V(N, p, q) (\sum_{\substack{\Delta\subset [0, 1]^2\\ l(\Delta)=N^{-1/2}}}\|E_{\Delta}g\|_{L^p(w_{B_N})}^q)^{1/q},
\endeq
holds for all measurable functions $g: [0, 1]^2\to \C$ and all balls $B_N\subset \R^5$ with radius $N$. An inequality of this form will be referred to as an $l^q L^p$ decoupling. We note that
$$V(N, p, q)=V^{(2,2)}_{p,q}(N^{-1/2}).$$

For the case $q=p$, Bourgain and Demeter \cite{BD2} proved the following sharp estimates
\begin{thm}\label{0626theorem1.1}
For each $p\ge 2$ and each $\epsilon>0$, there exists $C_{p, \epsilon}>0$ such that
\beq\label{0626e1.3}
 V(N, p, p)\le
 \begin{cases}
 \hfill C_{\epsilon, p} N^{\frac{1}{2}-\frac{1}{p}+\epsilon}, \hfill &  \text{if } 2\le p\le 8\\
 \hfill C_{\epsilon, p} N^{1-\frac{5}{p}+\epsilon}, \hfill  & \text{ if } p\ge 8.
 \end{cases}
\endeq
\end{thm}

This result follows via interpolation from the sharp estimate at the critical exponent
\begin{equation}
\label{abc1}
V(N,8,8)\lesssim_\epsilon N^{\frac12-\frac18+\epsilon}.
\end{equation}
It turns out that this estimate (i.e. with $q=8$) is not strong enough for our purposes. Instead, we need an $l^qL^8$ decoupling with a $q$ that matches the one from Lemma \ref{main1}. As remarked there, this forces the restriction $q<\frac{40}{9}$.

We will prove the following  stronger estimate.
\begin{thm}\label{0626theorem1.3}
For each $q\in [\frac{8}{3}, 8]$ and each $\epsilon>0$, there exists $C_{q, \epsilon}>0$ such that
\beq\label{0626e1.10}
V(N, 8, q)\le C_{q, \epsilon} N^{\frac{1}{2}-\frac{1}{q}+\epsilon}.
\endeq
Moreover, the power of $N$ is sharp for each $q$.
\end{thm}
Let us first get an idea about sharpness. As observed before, see for example Theorem 2.2 in \cite{BD2}, we have the following exponential sum estimate, valid for each $a_{n,m}\in\C$
$$\|\sum_{n=1}^{N}\sum_{m=1}^Na_{n,m}e(nx_1+mx_2+n^2x_3+nmx_4+m^2x_5)\|_{L^p([0,1]^5)}\lesssim V(N^2,p,q)\|a_{n,m}\|_{l^q}.$$
By taking $a_{n,m}\equiv 1$ and by restricting $|x_1|,|x_2|\ll N^{-1}$ and $|x_3|,  |x_4|,|x_5|\ll N^{-2}$, the left hand side is seen to be  $\gtrsim N^{2-\frac8p}$. When $p=8$ this leads to the lower bound
$$V(N^2, 8, q)\gtrsim N^{1-\frac{2}{q}}.$$

In all fairness, we will only use  estimate \eqref{0626e1.10} later with $q=\frac{8}{3}$. To show the key differences with the case $q=8$, it will help to present things in the slightly larger generality $\frac83\le q\le 8$. The fact that \eqref{0626e1.10} implies \eqref{abc1} follows from H\"older's inequality. It is not difficult to see that \eqref{0626e1.10} fails for $2\le q<\frac83$. For each such $q$ there will be two critical exponents $p$, as opposed to just one for $\frac83\le q\le 8$ ($p=8$). We hope to address this new phenomenon elsewhere.

We will prove Theorem \ref{0626theorem1.3} in the remainder of the section.

\subsection{The Brascamp-Lieb inequality and a transversality condition}\label{0705subsection1.2}

Let $M$ be a positive integer. For $1\le j\le M$, let $V_j$ be a $d_0-$dimensional linear subspace of $\R^n$.  Define
\beq
\Lambda(f_1, f_2, ..., f_M)=\int_{\R^n} \prod_{j=1}^M f_j (\pi_j (x))dx,
\endeq
for $f_j: V_j\to \C$. Here $\pi_j$ denotes the orthogonal projection  onto $V_j$. We recall the following theorem from Bennett, Carbery, Christ and Tao \cite{BCCT}.
\begin{thm}(\cite{BCCT})
\label{abc14}
Given $p\ge 1$, the estimate
\beq\label{1803e1.8}
|\Lambda(f_1, f_2, ..., f_M)| \lesim \prod_{j=1}^M \|f_j\|_{p},
\endeq
holds for arbitrary $f_j\in L^p(V_j)$ if and only if $np=d_0M$ and the following Brascamp-Lieb transversality condition is satisfied
\beq\label{1803e1.9}
dim(V) \le \frac{n}{d_0M} \sum_{j=1}^M dim(\pi_j(V)), \text{ for each (linear) subspace } V\subset \R^n.
\endeq
\end{thm}
Let $$n_1(t,s)=(1,0,2t, s,0),$$
$$n_2(t,s)=(0,1,0,t,2s).$$
Now let us be more specific about $d_0, n$ and $M$.  Throughout this section, we will work with $n=5$, which comes from the fact that we are considering the surface $\mc{S}_{2,2}$ in $\R^5$. We will take $d_0=2$, since the tangent space $\langle n_1(t,s),n_2(t,s)\rangle$ to the  surface $\mc{S}_{2,2}$ has dimension two. The degree $M$ of multilinearity will be variable.

Under these choices of various parameters, condition \eqref{1803e1.9} becomes
\beq\label{1712e1.13}
dim(V)\le \frac{5}{2M}\sum_{j=1}^{M} dim(\pi_j(V)), \text{ for each subspace } V\subset \R^5.
\endeq

As explained in \cite{BD2},  Theorem \ref{abc14} has a whole host of consequences. It first leads to an appropriate multilinear restriction theorem. This in turn leads to the multi-scale inequality Proposition \ref{gjityiophjytophpotigirti0-we=fdcwee=w=}. We refer the reader to \cite{BD2} for details.
\medskip

We next recall the concept and relevant properties of transversality from \cite{BD2}.
Given a polynomial function $Q(t,s)$ of any degree $\deg(Q)$, denote by $\|Q\|$ the $l^1$ norm of its coefficients.
\begin{defi}
\label{dfek1}
A collection consisting of $m\ge 5$ sets $S_1,\ldots,S_{m}\subset [0,1]^2$ is said to be $\nu-$transverse (for $\mc{S}_{2,2}$) if the following  requirement is satisfied:
\medskip

For each $1\le i_1\not= i_2\ldots\not= i_{\left[\frac{m}{5}\right]+1}\le m$  we have
\begin{equation}
\label{fek4}
\inf_{\deg(Q)\le 2,\atop{\|Q\|=1}}\max_{1\le j\le \left[\frac{m}{5}\right]+1}\inf_{(t,s)\in S_{i_j}}|Q(t,s)|\ge \nu.
\end{equation}
\end{defi}
 Note that transverse sets are not necessarily pairwise disjoint. Requirement \eqref{fek4} says that $\left[\frac{m}{5}\right]+1$ points in  different sets $S_i$ do not come "close`` to belonging to the zero set of a (nontrivial) polynomial function $Q$ of degree $\le 2$. The following is Proposition 4.4 ($k=2$) from \cite{BD2}.

\begin{prop}
\label{pfek1}
Consider $m\ge 5$ points $(t_j,s_j)\in [0,1]^2$ such that the sets $S_j=\{(t_j,s_j)\}$ are $\nu-$transverse  for some $\nu>0$. Then the $m$ planes $V_j,\;1\le j\le m$  spanned by the vectors $n_1(t_j,s_j)$ and $n_2(t_j,s_j)$ in $\R^n$ satisfy requirement \eqref{1712e1.13}.
\end{prop}

A $K$-square is defined to be a closed square of side length $1/K$ inside $[0, 1]^2$. We may work with $K$ among integer powers of $2$. The collection of all dyadic $K$-squares will be denoted by $Col_K$. A standard compactness argument leads to the following nice criterium for transversality of squares in $Col_K$.

\begin{lem}\label{0604lemma2.4}
Let $R_1, ..., R_{m}\subset [0, 1]^2$ be $m\ge 5$ squares in $Col_K$ such that given any polynomial $Q(t,s)$ of degree one or two,  at most $\left[\frac{m}{5}\right]$ of them intersect the $\frac{10}{K}$ neighborhood of the zero set of $Q$.
Then there exists $\nu_K$ depending only on $K$ such that $R_1, ..., R_{m}$ are $\nu_K$-transverse.
\end{lem}

\subsection{Multilinear and linear decouplings}\label{section3}
In this section, we run the Bourgain--Guth argument from \cite{BG} to show that the linear decoupling constant $V(N, p,q)$ is equivalent to a certain multilinear decoupling constant.  Let $V_{multi}(N, p, q,  \nu)$ be the smallest constant such that the inequality
\beq
\|(\prod_{i=1}^{M}E_{R_i}g_i)^{\frac{1}{M}}\|_{L^p(\w)}\le V_{multi}(N, p, q,  \nu) (\prod_{i=1}^{M}\sum_{\substack{\Delta\subset R_i\\l(\Delta)=N^{-1/2}}}\|E_{\Delta}g_i\|_{L^p(\w)}^q)^{\frac{1}{Mq}}
\endeq
holds for each $M$ squares $R_j$  that are $\nu$-transverse for some $\nu>0$.
By H\"older's inequality, one can see immediately that for each $p,q$
\beq
\label{abc4}
V_{multi}(N, p, q, \nu)\le V(N, p,q).
\endeq
We will prove that the reverse direction of the above inequality is also essentially true in a certain  range for $p,q$, quantified as follows.
\begin{thm}\label{2903thm3.1}
For $\;p\ge 6$ and $q\ge 2$ with $p\ge q$, $\epsilon>0$ and $K\ge 1$, there exist constants $\beta(K, p, q, \epsilon)$ and $\Lambda_{K, p, q, \epsilon}$ with
\beq
\lim_{K\to \infty}\beta(K, p, q, \epsilon)=0
\endeq
such that for each $N\ge K\ge 1$,
\beq\label{0705e1.27}
\begin{split}
& V(N, p, q) \le N^{\beta(K, p, q, \epsilon)+\frac{1}{2}-\frac1{2q}-\frac{3}{2p}+\epsilon}\\
&+ \Lambda_{K, p, q, \epsilon} \log_K N \max_{1\le M\le N}\big[ (\frac{M}{N})^{-\frac{1}{2}+\frac1{2q}+\frac{3}{2p}+\epsilon} V_{multi}(M, p, q, \nu_K) \big].
\end{split}
\endeq
Here $\nu_K$ is the transversality constant depending on $K$, coming from Lemma \ref{0604lemma2.4}.
\end{thm}
When $p=q$, this is Theorem 8.1 from \cite{BD2}. The proof for the other values of $p,q$ will be very similar, following the Bourgain--Guth original insight from \cite{BG}. One needs to deal with a lower dimensional contribution and with a multilinear transverse term. The only key difference in our argument here, compared with the one in Theorem 8.1 from \cite{BD2}, is in the way we estimate the lower dimensional term. A trivial decoupling sufficed in \cite{BD2}, while here we need to invoke the more sophisticated decoupling for the parabola. See Theorem \ref{abc2} and Claim \ref{2903claim3.3} below.

The equivalence between $V(N, p,q)$ and $V_{multi}(N, p, q,\nu_K)$ in the estimate \eqref{0705e1.27} can be interpreted in the following way. Let us focus for simplicity on the range $p=8$, $q> \frac83$ which is relevant for Theorem \ref{0626theorem1.3}. Note that in this range we have
$$\lambda_{1,q}:=\frac{1}{2}-\frac1{2q}-\frac{3}{2\cdot 8}=\frac5{16}-\frac1{2q}< \frac12-\frac1q:=\lambda_{2,q}.$$

As remarked after Theorem \ref{0626theorem1.3}, we have the lower bound
$$V(N,8,q)\gtrsim N^{\lambda_{2,q}}.$$
Combining these two estimates, the inequality in Theorem \ref{2903thm3.1} can be simplified (for $K$ large enough, so that $\beta$ is small enough) as follows
\begin{equation}
\label{abc5}
V(N, 8, q) \le \Lambda_{K, 8, q, \epsilon} \log_K N \max_{1\le M\le N}\big[ (\frac{N}{M})^{\lambda_{1,q}+\epsilon} V_{multi}(M, 8, q, \nu_K)\big].
\end{equation}
Write $V_{multi}(N, 8, q, \nu_K)\sim N^{\lambda_{3,q}}$. It can not be that $\lambda_{3,q}\le \lambda_{1,q}$, as \eqref{abc5} would then lead to the contradiction
$$V(N,8,q)\lesssim_\epsilon N^{\lambda_{1,q}+\epsilon}.$$
So it must be that $\lambda_{3,q}> \lambda_{1,q}$ in which case \eqref{abc5} implies that
$$V(N, 8, q) \le \Lambda_{K, 8, q, \epsilon} (\log_K N)V_{multi}(M, 8, q, \nu_K).$$
Ignoring the logarithmic term, we may view this inequality as a reverse inequality for \eqref{abc4}.

\bigskip

Before starting the proof of Theorem \ref{2903thm3.1}, we state an auxiliary result. Let $V^{(1,2)}(N,p,q)$ denote the decoupling constants associated with the parabola $\mc{S}_{1,2}$. More precisely, $V^{(1,2)}(N,p,q)$ is the smallest constant such that the following inequality  holds true
$$
\|E_{[0, 1]}^{(1,2)}g\|_{L^p(w_{B_N})}\le V^{(1,2)}(N, p, q) (\sum_{\substack{I\subset [0, 1]\\ l(I)=N^{-1/2}}}\|E_{I}^{(1,2)}g\|_{L^p(w_{B_N})}^q)^{1/q},
$$
for all $g:[0,1]\to\C$ and all balls $B_N$ in $\R^2$ with radius $N$.
\begin{thm}
\label{abc2}
For each $p\ge 6$ and each $q\ge 2$ we have
$$V^{(1,2)}(N, p, q)\lesssim_\epsilon N^{\frac12(1-\frac1q-\frac3p)+\epsilon}.$$
\end{thm}
\begin{proof}
The estimate when $q=2$ was proved in \cite{BD1}. The estimate for $q\ge 2$ follows from this and H\"older's inequality.
\end{proof}

Theorem \ref{2903thm3.1} will be obtained by iterating the following inequality.

\begin{prop}\label{2903p3.2}
For $p\ge 6$ and $q\ge 2$ with $p\ge q$,  for each $\epsilon>0$ and $N\ge K\ge 1$, we have
\beq
\begin{split}
\label{abc3}
& \|E_{[0, 1]^2} g\|_{L^p(\w)}\lesim_{\epsilon, p}  (\sum_{R\in Col_K}\|E_{R}g\|_{L^p(w_{B_N})}^q)^{1/q}\\
& +K^{\frac12(1-\frac{1}{q}-\frac{3}{p})+\epsilon} (\sum_{\beta\in Col_{K^{1/2}}}\|E_{\beta}g\|_{L^p(w_{B_N})}^q)^{1/q}\\
&+ K^{100} V_{multi}(N, p, q, \nu_K) (\sum_{\Delta\in Col_{N^{1/2}}}\|E_{\Delta}g\|_{L^p(w_{B_N})}^q)^{1/q}
\end{split}
\endeq
\end{prop}
\begin{proof}[Proof of Proposition \ref{2903p3.2}]
We split $[0, 1]^2$ into squares of side length $K^{-1}$, and write
\beq
E_{[0, 1]^2}g=\sum_{R\in Col_K} E_{R}g.
\endeq
By the uncertainty principle, on each ball $B_K$ of radius $K$, the function $|E_R g|$ is essentially a constant. We use $|E_R g(B_K)|$ to denote this constant, and we write
$|E_R g(x)|\approx |E_R g(B_K)|$ for $x\in B_K$. We now fix $B_K$ for a while. Denote by $R^*=R^*(B_K)$ the square that maximises $|E_{R}g(B_K)|$. Let $Col_K^*$ be those squares  $R\in Col_K$ such that
\beq\label{0705e1.31}
|E_R g(B_K)|\ge K^{-2} |E_{R^*}g(B_K)|.
\endeq
Initialise
$$
STOCK=Col^*_K
$$
We repeat the following algorithm.
\medskip

Case 1: If $|STOCK|\le 4$, then the algorithm stops, after performing the following computations.
We can write on each $x\in B_K$
$$|E_{[0, 1]^2} g(x)|=|\sum_{R\in Col_K} E_{R}g(x)|$$$$\le \sum_{R\not\in Col^*_K}|E_Rg(x)|+|\sum_{R\in Col^*_K} E_{R}g(x)|$$
$$\le 5\max_R|E_{R}g(B_K)|.$$
Integrating this on $B_K$   leads to
$$\|E_{[0, 1]^2} g\|_{L^p(w_{B_K})}\lesim (\sum_{R\in Col_K}\|E_{R}g\|_{L^p(w_{B_K})}^q)^{1/q}.$$
Then raise to the power $p$, sum over a finitely overlapping cover of $B_N$ using balls $B_K$ and invoke Minkowski's inequality (using that $p\ge q$) to recover the desired \eqref{abc3}. In fact, note that only the first  term on the right hand side of \eqref{abc3} is needed in this case.
\medskip

Case 2: If $M:=|STOCK|\ge 5$ and if   given any polynomial $Q(t,s)$ of degree one or two,  at most $\left[\frac{M}{5}\right]$ of the squares in $STOCK$ intersect the $\frac{10}{K}$ neighborhood of the zero set of $Q$, then the algorithm stops, after performing the following computations.  Note first that in this case Lemma \ref{0604lemma2.4} guarantees that the squares in $STOCK$ are $\nu_K-$transverse.
Thus, by \eqref{0705e1.31} and the triangle inequality, we have for $x\in B_K$
$$
|E_{[0, 1]^2}g(x)| \le K^2\max|E_{R}g(B_K)|\le K^{4} \Big(\prod_{i=1}^{M}|E_{R_i}g(B_K)| \Big)^{1/M}.
$$
Integrating on $B_K$, then raising to the power $p$ and summing over $B_K$ as before leads to the inequality \eqref{abc3}. Note that only the third term is needed this time.

\medskip

\medskip

Case 3: Assume $M:=|STOCK|\ge 5$ and that there is a polynomial $Q(t,s)$ of degree one or two, and a subset $\mc{G}\subset STOCK$ with at least $\left[\frac{M}{5}\right]+1$  squares, each of which  intersects the $\frac{10}{K}$ neighborhood of the zero set of $Q$.
We denote by $\mc{G}_{K^{\frac{1}{2}}}$ the collection of the squares $\beta$ from $Col_{K^{\frac{1}{2}}}$ which contain at least one element from $\mc{G}$. Note that each square in $\mc{G}_{K^{\frac{1}{2}}}$ will be inside the $10K^{-\frac{1}{2}}$ neighbourhood of the zero set of  $Q$. We write
\beq
\label{abcd1}
|E_{[0,1]^2}g|\le |\sum_{\beta\in \mc{G}_{K^{\frac{1}{2}}}} E_{\beta}g|+|\sum_{\beta\notin\mc{G}_{K^{\frac{1}{2}}}} E_{\beta}g|.
\endeq
We reset the value
$$STOCK:=STOCK\setminus\{R:\;R\subset \beta,\text{ for some }\beta\in \mc{G}_{K^{\frac{1}{2}}}.\}$$
and repeat the algorithm.
\medskip

The only interesting discussion left is about what happens if the algorithm is repeated a few times. Note first that the only way to be repeated is if each time we end up with Case 3. Second, note that $STOCK$ looses at least one fifth of its size after each repetition of the algorithm. Since $STOCK$ has size at most $K^2$ in the beginning, it follows that the algorithm can only be repeated $O(\log K)$ times. Each repetition will add another term to the sum \eqref{abcd1}. Each such term will be estimated using the following result.

\begin{claim}\label{2903claim3.3}
Let $\mc{G}_{K^{\frac{1}{2}}}$ be a subcollection of $Col_{K^{\frac{1}{2}}}$ consisting of squares that are subsets of the ${10}{K^{-1/2}}$ neighbourhood of the zero set of  $Q$.
Then for each $p\ge 6$ and $q\ge 2$ we have
\beq\label{0705e1.37}
\|\sum_{\beta\in \mc{G}_{K^{\frac{1}{2}}}} E_{\beta}g\|_{L^p(w_{B_K})} \lesim_{p, \epsilon} K^{\frac{1}{2}-\frac1{2q}-\frac{3}{2p}+\epsilon} (\sum_{\beta\in \mc{G}_{K^{\frac{1}{2}}}} \|E_{\beta}g\|_{L^p(w_{B_K})}^q)^{1/q},
\endeq
for each $\epsilon>0$.
\end{claim}
\begin{proof}[Proof of Claim \ref{2903claim3.3}]
The zero set of $Q$ is the union of $O(1)$ points and curves, each of which can be thought of as the graph of a function with $O(1)$ derivative. It suffices to assume that the zero set of $Q$ is one such curve given by $s=l(r)$, with $\|l'\|_\infty\lesssim 1$. If not, repeat the following argument with the roles of $r,s$ reversed.  At the expense of an $O(1)$ loss, we may thus assume that there is at most one square $\beta\in \mc{G}_{K^{\frac{1}{2}}}$ in each strip $I\times [0,1]$ with $|I|=K^{-1/2}$. Denote by $U\subset [0,1]^2$ the union of these squares, and write $g_U=g1_U$. Then note that
$$\sum_{\beta\in \mc{G}_{K^{\frac{1}{2}}}} E_{\beta}g=E_{[0,1]^2}g_U$$
and that
$$E_{\beta}g=E_{I_\beta\times [0,1]}g_U$$
where $I_\beta$ is the projection onto the $r$ axis of $\beta$. These observations allow us to recast the claimed inequality \eqref{0705e1.37} in the following more convenient form
$$
\|E_{[0,1]^2}g_U\|_{L^p(w_{B_K})} \lesim_{p, \epsilon} K^{\frac{1}{2}-\frac1{2q}-\frac{3}{2p}+\epsilon} (\sum_{|I|=K^{-1/2}} \|E_{I\times [0,1]}g_U\|_{L^p(w_{B_K})}^q)^{1/q}.
$$
But this inequality follows immediately from Theorem \ref{abc2} ($N=K$) combined with (a weighted version of) Corollary \ref{abc9'} with $\mc{M}_2=\mc{S}_{2,2}$ and $Q_1(r)=r^2$.

\end{proof}
By invoking the above Claim (absorbing the $\log K$ term into the $K^{\epsilon}$ term),  the contribution from Case 3 can be estimated by
$$\|E_{[0, 1]^2} g\|_{L^p(w_{B_K})}\lesim_{\epsilon, p}K^{\frac12(1-\frac{1}{q}-\frac{3}{p})+\epsilon} (\sum_{\beta\in Col_{K^{1/2}}}\|E_{\beta}g\|_{L^p(w_{B_K})}^q)^{1/q}.$$
As before, then raise this inequality  to the power $p$, sum over a finitely overlapping cover of $B_N$ using balls $B_K$ and invoke Minkowski's inequality  to recover the desired \eqref{abc3}.
\bigskip

Of course, in reality, our selection algorithm may run differently, depending on $B_K$. This case can be dealt with by combining the analysis from Case 1, Case 2 and Case 3.

\end{proof}

To obtain  Theorem \ref{2903thm3.1}, we iterate ( a rescaled version of) Proposition \ref{2903p3.2} using parabolic rescaling. All  details are in Section 8 from \cite{BD2}.

\subsection{The final iteration }\label{section4}
In this section we finalize the proof of Theorem \ref{0626theorem1.3}.
We start by recalling Corollary 6.7 from \cite{BD2}. For $p>5$ we let
$$\kappa_p=\frac{p-5}{p-2}.$$

\begin{prop}
\label{gjityiophjytophpotigirti0-we=fdcwee=w=}
For each $\nu>0$, $p\ge 5$ and $\epsilon>0$ there exists a constant $C_{p,\nu,\epsilon}$ such that for each $\kappa_p\le \kappa\le 1$,  for each $\nu-$transverse squares $R_1,\ldots,R_{M}\subset [0,1]^2$, each ball $B_R$ in $\R^n$ with radius $R\ge N\ge 1$ and each $g_i:R_i\to \C$ we have
$$\|(\prod_{i=1}^{M}\sum_{\atop{l(\tau)=N^{-1/2}}}|E_{\tau}g_i|^2)^{\frac1{2M}}\|_{L^{p}(w_{B_R})}\le $$$$C_{p,\nu,\epsilon}
N^{\epsilon}\|(\prod_{i=1}^{M}\sum_{\atop{l(\Delta)=N^{-1}}}|E_{\Delta}g_i|^2)^{\frac1{2M}}
\|_{L^{p}(w_{B_R})}^{1-\kappa}
(\prod_{i=1}^{M}\sum_{\atop{l(\tau)=N^{-1/2}}}\|E_{\tau}g_i\|_{L^{p}(w_{B_R})}^2)^{\frac{\kappa}{2M}},
$$
\end{prop}

Combing this with H\"older's inequality, we get the following inequality  suitable for iterations ($q\ge 2$, $l\ge 0$)
$$\|(\prod_{i=1}^{M}\sum_{\atop{l(\tau)=N^{-2^{-l}}}}|E_{\tau}g_i|^2)^{\frac1{2M}}\|_{L^{p}(w_{B_N})} \le C_{p, \nu,\epsilon}N^{\frac{\kappa}{2^{l-1}}(\frac12-\frac1q)+\epsilon}\times $$
\begin{equation}
\label{fe21}
 \|(\prod_{i=1}^{M}\sum_{\atop{l(\Delta)=N^{-2^{-l+1}}}}
|E_{\Delta}g_i|^2)^{\frac1{2M}}\|_{L^{p}(w_{B_N})}^{1-\kappa}(\prod_{i=1}^{M}\sum_{\atop{l(\tau)=N^{-2^{-l}}}}
\|E_{\tau}g_i\|_{L^{p}(w_{B_N})}^q)^{\frac{\kappa}{Mq}}.
\end{equation}

\bigskip

The following lemma is an immediate consequence of Cauchy--Schwarz, and has nothing to do with transversality. It simply serves as a starting point for our iteration.
\begin{lem}
\label{wlem0081}Consider $M$  squares $R_1,\ldots,R_{M}$. Assume $g_i$ is supported on $R_i$.
Then for $1\le p\le\infty$ and $s\ge 2$
$$\|(\prod_{i=1}^{M}|E_{R_i}g_i|)^{\frac1{M}}\|_{L^{p}({w_{B_N}})}\le N^{2^{-s}}\|(\prod_{i=1}^{M}\sum_{\atop{l(\tau_s)=N^{-2^{-s}}}}|E_{\tau_s}g_i|^2)^{\frac1{2M}}\|_{L^{p}(w_{B_N})}.$$
\end{lem}

\bigskip

We will apply these results with $p=8$ and $\frac83\le q\le 8$. Note that $\kappa_8=\frac12$.
For $\epsilon>0$, $K\ge 2$ we will from now on write $C_{K,\epsilon}$ for $C_{8,\nu_K,\epsilon}$.
\bigskip

Let $R_1,\ldots, R_{M}\in Col_K$ be arbitrary $\nu_K-$transverse squares and assume $g_i$ is supported on $R_i$.

Start with Lemma \ref{wlem0081}, continue with iterating \eqref{fe21} $s$ times, and invoke parabolic rescaling (Lemma \ref{abc18}) at each step to write for each $\frac12\le \kappa\le 1$
$$\|(\prod_{i=1}^{M}|E_{R_i}g_i|)^{\frac1{M}}\|_{L^8({B_N})}\le  N^{2^{-s}}\|(\prod_{i=1}^{M}\sum_{\atop{l(\tau_s)=N^{-2^{-s}}}}|E_{\tau_s}g_i|^2)^{\frac1{2M}}\|_{L^8(w_{B_N})}\le $$
$$ N^{2^{-s}}(C_{K,\epsilon}N^\epsilon)^{s}\prod_{l=1}^s
N^{\frac{\kappa}{2^{l-1}}(1-\kappa)^{s-l}(\frac12-\frac1q)}\times$$$$\|(\prod_{i=1}^{M}\sum_{\atop{l(\tau)=N^{-1}}}|E_{\tau}g_i|^2)^{\frac1{2M}}\|_{L^8(w_{B_N})}^{(1-\kappa)^{s}}\times $$
$$\prod_{i=1}^{M}\left[\prod_{l=1}^s(\sum_{\atop{l(\tau)=N^{-2^{-l}}}}\|E_{\tau}g_i\|_{L^8(w_{B_N})}^q)^{\frac{\kappa}{q}(1-\kappa)^{s-l}}\right]^{\frac1{M}}$$

$$ \le N^{2^{-s}}(C_{K,\epsilon}N^\epsilon)^{s}(\prod_{i=1}^{M}\sum_{\atop{l(\Delta)=N^{-1/2}}}
\|E_{\Delta}g_i\|_{L^8(w_{B_N})}^q)^{\frac{1-{(1-\kappa)^{s}}}{Mq}}\times$$
$$N^{2^{1-s}\kappa(\frac12-\frac1q)\frac{1-[2(1-\kappa)]^{s}}{1-2(1-\kappa)}}\times\|(\prod_{i=1}^{M}\sum_{\atop{l(\tau)=N^{-1}}}|E_{\tau}g_i|^2)^{\frac1{2M}}\|_{L^8(w_{B_N})}^{(1-\kappa)^{s}}\times$$
\begin{equation}
\label{fe34}
 V(N^{1-2^{-s+1}},8,q)^{\kappa}V(N^{1-2^{-s+2}},8,q)^{\kappa(1-\kappa)}\cdot\ldots\cdot V(N^{0},8,q)^{\kappa(1-\kappa)^{s-1}}.
\end{equation}

Note that the (very weak) inequality
$$\|(\sum_{\atop{l(\Delta)=N^{-1}}}|E_{\Delta}g_i|^2)^{\frac1{2}}\|_{L^8(w_{B_N})}\le N^{O(1)}(\sum_{\atop{l(\Delta)=N^{-1/2}}}\|E_{\Delta}g_i\|_{L^8(w_{B_N})}^q)^{1/q}$$
is a consequence of Minkowski's inequality and standard truncation arguments. The precise value of the exponent is not relevant, all that matters is that it is $O(1)$. Applying H\"older's inequality leads to

$$\|\prod_{i=1}^{M}(\sum_{\atop{l(\Delta)=N^{-1}}}|E_{\Delta}g_i|^2)^{\frac1{2M}}\|_{L^8(w_{B_N})}\le N^{O(1)}(\prod_{i=1}^{M}\sum_{\atop{l(\Delta)=N^{-1/2}}}\|E_{\Delta}g_i\|_{L^8(w_{B_N})}^q)^{\frac{1}{Mq}}.$$

Inserting this into \eqref{fe34} and maximizing over all choices of $R_i\in Col_K$ which are $\nu_K$ transverse, \eqref{fe34} has the following consequence, for all $N\ge K$
$$V_{multi}(N,8,q,\nu_K)\le (C_{K,\epsilon}N^\epsilon)^{s-1} N^{2^{-s}}N^{2^{1-s}\kappa(\frac12-\frac1q)\frac{1-[2(1-\kappa)]^{s}}{1-2(1-\kappa)}}\times$$
\begin{equation}
\label{fe23}
V(N^{1-2^{-s+1}},8,q)^{\kappa}V(N^{1-2^{-s+2}},8,q)^{\kappa(1-\kappa)}\cdot\ldots\cdot V(N^{0},8,q)^{\kappa(1-\kappa)^{s-1}}  N^{O((1-\kappa)^s)}.
\end{equation}
\bigskip
Let $\gamma_q$ be the unique positive number such that
$$\lim_{N\to\infty}\frac{V(N,8,q)}{N^{\gamma_q+\delta}}=0,\;\text{for each }\delta>0$$
and
\begin{equation}
\label{fe27}
\limsup_{N\to\infty}\frac{V(N,8,q)}{N^{\gamma_q-\delta}}=\infty,\;\text{for each }\delta>0.
\end{equation}
By using the fact that $V(N,8,q)\lesssim_\delta N^{\gamma_q+\delta}$ in \eqref{fe23}, it follows that for each $\delta,\epsilon>0$ and $K, s\ge 2$
\begin{equation}
\label{fe32}
\limsup_{N\to\infty}\frac{V_{multi}(N,8,q,\nu_K)}{N^{\gamma_{q,\delta,s,\epsilon,\kappa}}}<\infty
\end{equation}
 where
$$\gamma_{q,\delta,s,\epsilon,\kappa}=\epsilon(s-1)+2^{-s}+\kappa(\gamma_q+\delta)(\frac{1-(1-\kappa)^{s}}{\kappa}-2^{-s+1}
\frac{1-[2(1-\kappa)]^{s}}{2\kappa-1})+$$
\begin{equation}
\label{nmcvuyfgurivgnyy35t789t89y890}
2^{1-s}\kappa(\frac12-\frac1q)\frac{1-[2(1-\kappa)]^{s}}{2\kappa-1} +O_p((1-\kappa)^s).
\end{equation}
Recall that our goal is to prove that
\begin{equation}
\label{djcfhughyvg7trgyn7otgy7nn845nyon}
\gamma_{q}\le \frac12-\frac1{q}.
\end{equation}
\bigskip

Assume for contradiction that this is not true, for some fixed $q\in[\frac83,8]$. Then, for $\kappa$ larger than but close enough to $\frac12$  we have
\begin{equation}
\label{fe26}
\gamma_q>\frac{2\kappa-1}{2\kappa}+\frac12-\frac1q.
\end{equation}
Note that \eqref{fe32} holds for this $\kappa$, as $\kappa_8=\frac12$. A simple computation using that $2(1-\kappa)<1$ and \eqref{fe26} shows that for $s$ large enough and for $\epsilon,\delta$ small enough we have
\begin{equation}
\label{fe33}
\gamma_{q,\delta,s,\epsilon,\kappa}<\gamma_q.
\end{equation}
Indeed, this follows easily by noticing that \eqref{nmcvuyfgurivgnyy35t789t89y890} implies
$$2^s(\gamma_{q,\delta,s,\epsilon,\kappa}-\gamma_q)=o_{\epsilon,\delta,s}(1)+1+\frac{2\kappa}{2\kappa-1}(\frac12-\frac1q-\gamma_q).$$
Fix such  $\epsilon,\delta, s,\kappa$.

Recalling \eqref{0705e1.27},  for each $N\ge K$ we have
\begin{equation}
\label{fek16}
\begin{split}
& V(N, 8, q) \le N^{\beta(K, 8, q, \epsilon)+\frac{1}{2}-\frac1{2q}-\frac{3}{16}+\epsilon}\\
&+ \Lambda_{K, 8, q, \epsilon} \log_K N \max_{1\le M\le N}\big[ (\frac{M}{N})^{-\frac{1}{2}+\frac1{2q}+\frac{3}{16}+\epsilon} V_{multi}(M, 8, q, \nu_K) \big].
.\end{split}
\end{equation}

We  next argue that $\gamma_{q,\delta,s,\epsilon,\kappa}\le \frac{1}{2}-\frac1{2q}-\frac{3}{16}$. If this were not true, we could choose $K$ large enough so that $$\beta(K,8,q,\epsilon)+\frac{1}{2}-\frac1{2q}-\frac{3}{16}\le\gamma_{q,\delta,s,\epsilon,\kappa}.$$
Combining this with \eqref{fe32} and \eqref{fek16} leads to
$$D(N,8,q)\le( \Lambda_{K, 8, q, \epsilon} \log_K N+1)N^{\gamma_{q,\delta,s,\epsilon,\kappa}}.$$
This of course contradicts \eqref{fe33} and \eqref{fe27}.

Using now that $\gamma_{q,\delta,s,\epsilon,\kappa}\le \frac{1}{2}-\frac1{2q}-\frac{3}{16}$ together with \eqref{fe32}, we can rewrite  \eqref{fek16} as follows
$$D(N,8,q)\le N^{\beta(K,8,q,\epsilon)+\frac{1}{2}-\frac1{2q}-\frac{3}{16}}+ \Lambda_{K, 8, q, \epsilon} \log_K NN^{\frac{1}{2}-\frac1{2q}-\frac{3}{16}}.$$
Since choosing $K$ large sends $\beta(K,8,q,\epsilon)$ to zero, the above inequality forces $\gamma_q\le \frac{1}{2}-\frac1{2q}-\frac{3}{16}$.

It will now be crucial to observe that, due to our original restriction $q\ge \frac83$, we have
\begin{equation}
\label{abc12}
\frac{1}{2}-\frac1{2q}-\frac{3}{16}\le \frac12-\frac1q.
\end{equation}
This further forces $\gamma_q\le \frac12-\frac1q$, contradicting our original assumption that \eqref{djcfhughyvg7trgyn7otgy7nn845nyon} is false.

This very last line of the argument explains why we needed to enforce the restriction $q\ge \frac83$ in the inequality \eqref{0626e1.10}. Recall also that the other restriction, namely $q\le 8$ was needed for various applications of Minkowski's inequality.

\section{A three dimensional cubic surface}
\label{abcsec5}

Recall the function
\beq
\Phi^{(2, 3)}(r, s)=(r, s, r^2, rs, s^2, r^3, r^2 s, r s^2, s^3).
\endeq
and the surface
\beq
\mc{S}_{2, 3}:=\{\Phi^{(2, 3)}(r, s): (r, s)\in [0, 1]^2\}.
\endeq
Let $0<\delta\le 1$. Recall that we have denoted  by $V^{(2, 3)}_{p, q}(\delta)$ the smallest constant such that
\beq\label{0709e2.3}
\|E^{(2, 3)}_{[0, 1]^2} g\|_{L^p(w_B)} \le V^{(2, 3)}_{p, q}(\delta) (\sum_{\substack{R:\; \text{square in } [0, 1]^2;\\ l(R)=\delta}} \|E^{(2, 3)}_R g\|_{L^p(w_B)}^{q})^{1/q}
\endeq
holds for each ball $B\subset \R^9$ of radius $\delta^{-3}$ and each $g:[0,1]^2\to\C$. \\

In this section we will prove the second half of Theorem \ref{2405thm1.5}. More precisely, we will assume that
Conjecture \ref{conjecture0} holds for $d=2,k=3$ and $1\le l\le 2$,  and will prove the estimate \eqref{abc13}, which we recall below
\beq\label{3004e2.3}
V^{(2, 3)}_{20, 20}(\delta)\lesim_{\epsilon} \delta^{-2(\frac12-\frac1{20})+\epsilon}.
\endeq

When proving \eqref{3004e2.3}, we follow the general framework from Bourgain and Demeter \cite{BD2}  and Bourgain, Demeter and Guth \cite{BDG}. However there are several significant differences.

The first difference is that we iterate different quantities. In our proof, we iterate the quantity \eqref{0720g3.20}, which is
\beq\label{0720h3.5}
D_p (q, B^r):=\Big( \prod_{i=1}^M \sum_{J_{i, q}\subset R_i}\|E_{J_{i, q}} g\|^{\frac{8}{3}}_{L^p_{\#}(w_{B^r})} \Big)^{\frac{3}{8M}}.
\endeq
The choice of the exponent $q=\frac{8}{3}$ will be explained more thoroughly in Subsection \ref{0805sub2.1}. We recall that $q=2$ was used in \cite{BDG}.

A second difference with the proof in \cite{BDG} is in the multilinear Kakeya inequalities (see Lemma \ref{0720lemma3.5h}). To prove these inequalities, we need to check a transversality condition, as presented  in Lemma \ref{0711lemma1.3}. The case $d_0=2$ has been covered in Proposition 4.4 from \cite{BD2}, but the case $d_0=5$ needs some analysis.  \\

The organisation of this section is as follows. In Subsection \ref{0805sub2.1} we introduce transversality and the relevant multilinear Kakeya inequalities. These inequalities will be used to derive a crucial ball-inflation lemma (Lemma \ref{main1}).

In Subsection \ref{2405sub3.2}, we will introduce a multilinear decoupling inequality and will recall why it is  ``essentially equivalent'' to the linear one.

In Subsection \ref{2405sub3.3}, we will modify the iteration argument from  \cite{BDG} to prove \eqref{3004e2.3}.


\subsection{Transversality, Kakeya inequalities and a ball-inflation lemma}\label{0805sub2.1}
In the cubic case described in this section, we will use Theorem \ref{abc14}  with $n=9$ and both $d_0=2$ and $d_0=5$. Here $n=9$ is the dimension of the ambient space where the surface $\mc{S}_{2, 3}$ lives. The  $d_0=2$ case reflects the fact that the linear space spanned by the first order derivatives of $\Phi^{(2, 3)}$ has dimension two. The $d_0=5$ case reflects the fact  that the linear space spanned by the first and second order derivatives has dimension five. Moreover, $M$ will be a large constant that will be determined later.

\begin{defi}\label{0706defi2.2}
Let $M\ge 1000$. The $M$ sets $S_1, ..., S_M\subset [0, 1]^2$ are called $\nu$-transverse, if for each polynomial $P(r, s)$ with $\text{deg}(P)\le 100$ and $\|P\|= 1$, we have that for each choice of $\frac{M}{100}$ different sets $S_{i_1}, ..., S_{i_{\frac{M}{100}}}$, there exists at least one set $S_{i_j}$ such that
\beq
|P(r ,s)|\ge \nu, \text{ for each } (r, s)\in S_{i_j}.
\endeq
Here $\|P\|$ is given by the $l^1$ sum of all the coefficients of the polynomial $P$.
\end{defi}
In a qualitative way, $S_1, ..., S_M$ being transverse is the same as saying that for each polynomial $P(r, s)$ with $\text{deg}(P)\le 100$ and $\|P\|= 1$, the zero set of $P$ intersects  no more than $\frac{M}{100}$ sets from $S_1, ..., S_M$.

We make a remark that transverse sets need not be pairwise disjoint.
\\

For a point $(r, s)\in [0, 1]^2$, we introduce the following notation. We first let $\mc{M}^{(1)}(r, s)$ denote the $9\times 2$ matrix
with columns $\Phi^{(2, 3)}_r(r, s)$ and $\Phi^{(2, 3)}_s(r, s)$.
We also let $\mc{M}^{(2)}(r, s)$ denote the $9\times 5$ matrix with columns
$\Phi^{(2, 3)}_r(r, s), \Phi^{(2, 3)}_s(r, s), \Phi^{(2, 3)}_{rr}(r, s), \Phi^{(2, 3)}_{rs}(r, s), \Phi^{(2, 3)}_{ss}(r, s).$ We let $W^{(1)}(r, s)$ denote the two dimensional linear space
$$
W^{(1)}(r, s)=\langle \Phi^{(2, 3)}_r(r, s), \Phi^{(2, 3)}_s(r, s)\rangle.
$$
Finally, we let $W^{(2)}(r, s)$ denote the five dimensional linear space
$$
W^{(2)}(r, s)=\langle \Phi^{(2, 3)}_r(r, s), \Phi^{(2, 3)}_s(r, s), \Phi^{(2, 3)}_{rr}(r, s), \Phi^{(2, 3)}_{rs}(r, s), \Phi^{(2, 3)}_{ss}(r, s)\rangle.
$$

The transversality introduced in the above definition is stronger than the Brascamp-Lieb transversality condition \eqref{1803e1.9}. This is proven in the following lemma.
\begin{lem}\label{0711lemma1.3}
Let $\{(r_j, s_j)\}_{1\le j\le M}$ be $M$ different points from $[0, 1]^2$ which are also $\nu$-transverse for some $\nu>0$. Then the collection of linear spaces $W^{(1)}(r_j, s_j)$ satisfies the Brascamp-Lieb condition \eqref{1803e1.9} with $d_0=2$, and the collection of linear spaces $W^{(2)}(r_j, s_j)$ satisfies the Brascamp-Lieb condition \eqref{1803e1.9} with $d_0=5$.
\end{lem}

\begin{proof}[Proof of Lemma \ref{0711lemma1.3}:]  We only write down the details for the case $d_0=5$. The other case $d_0=2$ is similar, and was essentially dealt with in Proposition 4.4 from \cite{BD2}.

Fix a linear space $V\subset \R^{9}$. Let $dim(V)$ denote the dimension of the space $V$. Let $\{v_1, v_2, ..., v_{dim(V)}\}$ be an orthogonal basis of $V$. We need to show that
\beq\label{0713f1.11}
dim(V)\le \frac{9}{5M} \sum_{j=1}^M dim(\pi_j(V)).
\endeq
By the rank-nullity theorem, $dim(\pi_j(V))$ equals the rank of the matrix $\mc{M}_V^{(2)}(r_j, s_j)$, which is given by
\beq
(v_1, v_2, ..., v_{dim(V)})^T \times \mc{M}^{(2)}(r_j, s_j).
\endeq
By Conjecture \ref{conjecture0} the matrix $\mc{M}_V^{(2)}(r,s)$ has at least one minor  of order $\big[\frac{5\cdot dim(V)}{9}\big]+1$, whose determinant equals $P(r,s)$ for some nonzero polynomial $P$. Moreover, we know that the degree of the polynomial $P$ is smaller than
$$
2\cdot \left(\Big[\frac{dim(V)\cdot 5}{9}\Big]+1 \right)\le 100,
$$
since each entry of $\mc{M}_V^{(2)}(r, s)$ is a polynomial of degree at most two.
Recall that $\{(r_j, s_j)\}_{1\le j\le M}$ are $\nu$-transverse. By definition, we know that there exist at least $M(1-\frac{1}{1000})$ different points among these, on each of which the polynomial $P$ does not vanish. This is the same as saying that for these $(r_j, s_j)$, the matrix $\mc{M}_V^{(2)}(r_j,s_j)$ has rank at least
\beq
\Big[\frac{dim(V)\cdot 5}{9}\Big]+1.
\endeq
Hence the right hand side of \eqref{0713f1.11} is greater than
\beq
\frac{9}{5} (1-\frac{1}{1000}) \left( \Big[\frac{dim(V)\cdot 5}{9}\Big]+1 \right).
\endeq
The last display is easily seen to be bigger than or equal to $dim(V)$. This finishes the proof of the estimate \eqref{0713f1.11}.

\end{proof}

From the above lemma, we know that in order to apply the Brascamp--Lieb inequality, it suffices to guarantee the transversality introduced in Definition \ref{0706defi2.2}. Indeed, in the forthcoming Bourgain--Guth-type argument, we also need that the notion of transversality in Definition \ref{0706defi2.2} is ``stable''. To be precise, for $M$ different points in $[0, 1]^2$ which are transverse, we also need that all the points in a small neighbourhood of these points are transverse.

\begin{lem}\label{0711cor1.4}
There exists $\Lambda>0$ such that for each $K\ge 1$ there exists $\nu_{K}>0$   so that any $\Lambda K$ or more squares in $Col_K$  are $\nu_{K}-$transverse.
\end{lem}
\begin{proof}
Let $d\ge 1$. By the main theorem in \cite{Wo} it follows that the $\frac{10}{K}-$neighborhood in $[0,1]^2$ of the zero set of any nontrivial polynomial of degree $\le d$ in two variables will intersect at most $C_dK$ squares in $Col_K$. The quantity
$$\nu_K:=\min_{Col\subset Col_K\;\;\atop{|Col|\ge (C_{100}+1)K}}\inf_{\deg(Q)\le {100},\atop{\|Q\|=1}}\max_{R\in Col}\inf_{(t,s)\in R}|Q(t,s)|$$
is easily seen to be positive, via a compactness argument. We can take $\Lambda=(C_{100}+1)$
\end{proof}

%
%
Transversality will manifest itself in two ways throughout the argument. One is in the equivalence between linear and multilinear decoupling (see next subsection). The second manifestation is in the form of the following Kakeya inequality, essentially proved in \cite{BBFL}. This is a very close analog of Theorem 6.5 from \cite{BDG}, the proof is essentially identical to that one.
\begin{lem}\label{0720lemma3.5h}
Let $l=1$ or $l=2$ and define $d_0=\frac{l(l+3)}{2}$. Let $S_1, ..., S_M$ be sets in $[0, 1]^2$ that are $\nu$-transverse for some $\nu>0$. Consider $M$ families $\mc{P}_j$ consisting of rectangular boxes $P$ in $\R^9$, that we refer to as  plates,  having the following properties\\
1) For each $P\in \mc{P}_j$, there exits $(r_j, s_j)\in S_j$ such that $d_0$ of the axes of $P$ have side lengths equal to $R^{1/2}$ and span $W^{(l)}(r_j, s_j)$, while the remaining $9-d_0$ axes have side lengths equal to $R$;\\
2) all plates are subsets of a ball $B_{4R}$ of radius $4R$.\\
Then we have the following inequality
\beq
\nint_{B_{4R}} | \prod_{j=1}^M F_j |^{\frac{1}{2M}\frac{18}{d_0}} \lesim_{\epsilon,\nu} R^{\epsilon}\left[ \prod_{j=1}^M | \nint_{B_{4R}} F_j |^{\frac{1}{2M}} \right]^{\frac{18}{d_0}}
\endeq
for each function $F_j$ of the form
\beq
F_j=\sum_{P\in \mc{P}_j} c_P 1_P.
\endeq
\end{lem}

Now we are ready to state our main lemma, which will be referred to as the ``ball-inflation" lemma. This type of lemma first appeared in \cite{BDG}, and played a crucial role in proving the Vinogradov mean value theorem in dimension one.

\begin{lem}\label{main1}
Let $n=9$. Fix $l=1$ or $2$ and $p\ge \frac{16n}{3l(l+3)}$. Let $R_1, ..., R_M$ be $\nu$-transverse squares in $[0,1]^2$. Let $B$ be an arbitrary ball in $\R^n$ of radius $\rho^{-(l+1)}$. Let $\mc{B}$ be a finitely overlapping cover of $B$ with balls $\Delta$ of radius $\rho^{-l}$. For each $g:[0, 1]^2\to \C$, we have
\beq\label{0706e2.12}
\begin{split}
& \frac{1}{|\mc{B}|} \sum_{\Delta\in \mc{B}}\left[ \prod_{i=1}^M\left( \sum_{\substack{R'_i \text{ square in } R_i\\ l(R'_i)=\rho}} \|E_{R'_i}g\|_{L^{\frac{l(l+3)p}{2n}}_{\#}(w_{\Delta})}^\frac{8}{3} \right)^{\frac{3}{8}} \right]^{p/M}\\
& \lesim_{\epsilon,\nu} \rho^{-\epsilon} \left[ \prod_{i=1}^M\left( \sum_{\substack{R'_i \text{ square in } R_i\\ l(R'_i)=\rho}} \|E_{R'_i}g\|_{L^{\frac{l(l+3)p}{2n}}_{\#}(w_B)}^{\frac{8}{3}} \right)^{\frac{3}{8}} \right]^{p/M}.
\end{split}
\endeq
\end{lem}
\begin{proof}
The proof of this lemma is essentially the same as the proof of Theorem 6.6. in  \cite{BDG}. The constraint $p\ge \frac{16n}{3l(l+3)}$, which is the same as $\frac{l(l+3)p}{2n}\ge \frac{8}{3}$, corresponds to the constraint $p\ge 2n$ from Theorem 6.6. in \cite{BDG}. Under this constraint, one can apply H\"older's inequality
\beq
\label{erufguyrtgrt8g89rt07grt8wer00wfqe=dpqeopfifgi}
\begin{split}
& \Big( \sum_{\substack{R'_i \text{ square in } R_i\\ l(R'_i)=\rho}} \|E_{R'_i}g\|_{L^{\frac{l(l+3)p}{2n}}_{\#}(w_{\Delta})}^\frac{8}{3} \Big)^{\frac{3}{8}} \\
& \lesim (\#(R_i))^{\frac{3}{8}-\frac{2n}{l(l+3)p}}\Big( \sum_{\substack{R'_i \text{ square in } R_i\\ l(R'_i)=\rho}} \|E_{R'_i}g\|_{L^{\frac{l(l+3)p}{2n}}_{\#}(w_{\Delta})}^{\frac{l(l+3)p}{2n}} \Big)^{\frac{2n}{l(l+3)p}}.
\end{split}
\endeq
Here $\#(R_i)$ denotes the number of squares $R_i'$ inside $R_i$. Rather than redoing the rest of the argument, we invite the reader to take this as an exercise, upon reading the proof of Theorem 6.6 in  \cite{BDG}.

\end{proof}
This inequality will be used with $p$ very close to $20$.
The difference between our lemma and Theorem 6.6. in \cite{BDG} is rather subtle. The choice for the Lebesgue index $\frac{l(l+3)p}{2n}$ is not negotiable due to the nature of the argument. But there is some freedom in choosing the exponent $\frac83$. Let us explain.  In \cite{BDG} this exponent is chosen to be 2, because in the one dimensional case an $l^2L^p$ decoupling is proved. More precisely, the following is proved in \cite{BDG} for the twisted cubic $\mc{S}_{1, 3}$ at the critical exponent $p=12$
$$V^{(1,2)}_{12,2}(\delta)\lesssim_\epsilon \delta^{-\epsilon}.$$
The analogous inequality for $\mc{S}_{2, 3}$ at the critical exponent $p=20$
\begin{equation}
\label{abc16}
V^{(2,3)}_{20,2}(\delta)\lesssim_\epsilon \delta^{-\epsilon}
\end{equation}
is false. This is because $\mc{S}_{2, 3}$ contains the parabola $\mc{S}_{1, 2}$, whose critical index is $p=6$. The validity of \eqref{abc16} would force the estimate
$$V^{(1,2)}_{20,2}(\delta)\lesssim_\epsilon \delta^{-\epsilon},$$
which is known to be false ($20>6$).

Since we are eventually proving an $l^{20} L^{20}$ decoupling for $\mc{S}_{2, 3}$, one may wonder why not  use the index $q=20$ instead of $\frac83$ in \eqref{0706e2.12}. Recall that we will use  \eqref{0706e2.12} with $p$ (arbitrarily) close to $20$. The index $q$ that we use in place of $8/3$ needs to satisfy the restriction $\frac{l(l+3)p}{2n}\ge q$, in order for the proof of Lemma \ref{main1} to work. Indeed, this restriction allows for the critical use of H\"older's inequality in \eqref{erufguyrtgrt8g89rt07grt8wer00wfqe=dpqeopfifgi}. Plugging in the worst case scenario $l=1$, $n=9$, $p=20$ leads to the restriction
$\frac{40}{9}\ge q.$
On the other hand, it will become clear  that we need\footnote{In short, this is the restriction that appears in Theorem \ref{0626theorem1.3} } $q\ge \frac83$.
We could have thus made any choice $q\in[\frac83,\frac{40}{9}]$. We decided to work with $q=\frac83$ for no particular reason.

\subsection{Linear vs multilinear decoupling}\label{2405sub3.2}
Throughout the rest of the argument, we  will simplify notation and will write $V_p(\delta)$ for $V^{(2, 3)}_{p, p}(\delta)$, and also just $E$ for $E^{(2, 3)}$.

In this section we will recall a useful result from \cite{BD2}. Let us first introduce a multilinear version of the decoupling inequality \eqref{0709e2.3}. Recall $\Lambda$ from  Lemma \ref{0711cor1.4}. For $K$ large enough we denote by $V_{p}(\delta, K)$ the smallest constant such that
\beq\label{0720e3.20h}
\|(\prod_{i=1}^{\Lambda K} E_{R_i} g)^{1/\Lambda K}\|_{L^p(w_B)}\le V_{p}(\delta, K) \prod_{i=1}^{\Lambda K} (\sum_{R'\subset R_i:\; l(R'_i)=\delta} \|E_{R'_i} g\|_{L^p(w_B)}^{p} )^{\frac{1}{p\Lambda K}}.
\endeq
holds true for all distinct squares $R_i\in Col_K$, each ball  $B\subset \R^9$ of radius $\delta^{-3}$, and each $g:[0,1]^2\to\C$.
Next we recall Theorem 8.1 from \cite{BD2}.

\begin{thm}(\cite{BD2})
\label{abc37}
For each $p\ge 2$ and $K\in \N$, there exists $\Omega_{K, p}>0$ and $\beta(K, p)>0$ with
\beq
\lim_{K\to \infty} \beta(K, p)=0, \text{ for each }p,
\endeq
such that for each small enough $\delta$, we have
\beq
V_p(\delta)\le \delta^{-\beta(K, p)-2(\frac{1}{2}-\frac{1}{p})}+ \Omega_{K, p} \log_K \big(\frac{1}{\delta} \big)\max_{\delta \le \delta'\le 1}(\frac{\delta'}{\delta})^{2(\frac{1}{2}-\frac{1}{p})} V_p(\delta', K).
\endeq
\end{thm}
Note that this result is the analog of Theorem \ref{2903thm3.1} proved earlier in the paper. We briefly recall the argument from \cite{BD2}. One needs to deal with a lower dimensional term and with a multilinear transverse term.  Since we are dealing with $l^pL^p$ decouplings, the contribution of the lower dimensional term can be cheaply estimated using a trivial decoupling. This is unlike the case of Theorem \ref{2903thm3.1}, where a more sophisticated decoupling was needed.


\subsection{The proof of \eqref{3004e2.3}}\label{2405sub3.3}

Fix $\delta<1$ and $K\ge2$ for a while. Fix  also $\Lambda K$ squares $R_j\in Col_K$, with $\Lambda$ from  Lemma \ref{0711cor1.4}.

For a positive number $r$, we use $B^r$ to denote a ball of radius $\delta^{-r}$.  Define
\beq\label{0720g3.20}
D_t (q, B^r):=\Big( \prod_{i=1}^{\Lambda K} \sum_{R_{i, q}\subset R_i}\|E_{R_{i, q}} g\|^{\frac{8}{3}}_{L^t_{\#}(w_{B^r})} \Big)^{\frac{3}{8\Lambda K}}
\endeq
In the notation $R_{i, q}$, the index $i$ indicates that this square lies in $R_i$, and $q$ indicates that the square $R_{i, q}$ has side length $\delta^{q}$.

 Note that we use an $l^{\frac{8}{3}}$ instead of an $l^2$ summation. This is a subtle and significant departure from the  Bourgain--Demeter--Guth argument in \cite{BDG}.  The choice of $\frac83$ is made to match the exponent from Lemma \ref{main1}.

For $r>s$, we denote
\beq
A_p (q, B^r, s)=\Big( \frac{1}{|\mc{B}_s(B^r)|} \sum_{B^s\in \mc{B}_s(B^r)} D_2(q, B^s)^p \Big)^{1/p}.
\endeq
Here $\mc{B}_s(B^r)$ denotes a finitely overlapping cover of $B^r$ using balls $B^s$.
\bigskip

We will use the following rather immediate consequence of Minkowski's and H\"older's inequalities.
 \begin{lem}\label{0810lemma6.7h}
 \label{abc31}
 Let $\mc{B}$ be a finitely overlapping cover of a ball $B$ by smaller balls $B'$, in other words
 $$1_B\le \sum_{B'\in\mc{B}}1_{B'}\lesssim 1_B.$$
 Then for each $p\ge \frac83$
 \beq
 \label{abc34}
 \frac1{|\mc{B}|}\sum_{B'\in\mc{B}}D_p(q,B')^p\lesssim D_p(q,B)^p.\endeq
 Also
 \beq
 \label{abc35}
 A_p(q,B^r,s)\lesssim D_p(q,B^r).\endeq \end{lem}
 \begin{proof}
 For \eqref{abc34}, apply first the triangle inequality in $l^{\frac{3p}{8}}$ to write for each $i$
 $$\sum_{B'}(\sum_{R_{i, q}\subset R_i}\|E_{R_{i, q}} g\|^{\frac{8}{3}}_{L^p(w_{B'})})^{\frac{3p}{8}}\lesssim (\sum_{R_{i, q}\subset R_i}\|E_{R_{i, q}} g\|^{\frac{8}{3}}_{L^p(w_{B})})^{\frac{3p}{8}}.$$
 Next, the geometric average in $i$ is taken care of by using H\"older.

 Finally,  \eqref{abc35} will follow from \eqref{abc34} and the following consequence of H\"older $$D_2(q,B')\lesssim D_p(q,B').$$

\end{proof}

We will next indicate how to gradually decouple into frequency squares of smaller size at the same time as increasing the size of the spatial balls. There will be two types of iteration, that we will call $r-$iteration and $M-$iteration. We start by describing the overture of the $r-$iteration, which will involve two distinct ball inflations ($l=1$ and $l=2$). We will then show how to perform a typical  step of the iteration, using an $l=2$ ball inflation (more precisely, Lemma \ref{abc22}). The end product of the $r-$iteration will be recorded in inequality  \eqref{0709e2.60k}. We will then proceed with the $M-$iteration, which will lead to Proposition \ref{0712prop1.10}. In the end the argument, we will combine  Proposition \ref{0712prop1.10} with  Theorem \ref{abc37} to finalize the proof of the estimate $V_{20}(\delta)\lesssim_{\epsilon}\delta^{-\epsilon-2(\frac12-\frac1{20})}$.
\bigskip

The implicit assumption throughout the following calculations is that $p>\frac{72}5$. In reality, we will only need the results for $p$ arbitrarily close to $20$. The implicit constants will depend on $K$, since the squares we are using will be $\nu_K$ transverse.
\medskip

Define $\alpha_1, \alpha_2,\beta_2\in (0,1)$ as follows
$$
\frac{1}{\frac{2p}{9}}=\frac{\alpha_1}{\frac{5p}{9}}+\frac{1-\alpha_1}{2},
$$

$$
\frac{1}{\frac{5p}{9}}=\frac{\alpha_2}{p}+\frac{1-\alpha_2}{8},
$$

$$
\frac{1}{8}=\frac{1-\beta_2}{2}+\frac{\beta_2}{\frac{5p}{9}}.
$$
We will apply the following lemma in each step of the $r-$iteration, with $m$ of the form $2(\frac32)^s$, $s=0,1,2,\ldots$. Note that we start with squares of side length $\delta^{\frac{m}{2}}$ and end up with squares of smaller side length $\delta^{\frac{3m}{2}}$ and $\delta^{\frac{3m}{4}}$. The coefficient $\delta^{-C}$ is the correct one, in the sense that it is a product of only H\"older's inequality.
\begin{lem}
\label{abc22}
For $m\ge 2$ and each ball $B^{m\cdot\frac32}$ we have
$$\Big( \frac{1}{|\mc{B}_m(B^{m\cdot\frac32})|} \sum_{B^m\in \mc{B}_m(B^{m\cdot\frac32})} D_{\frac{5p}{9}}(\frac{m}2, B^m)^p \Big)^{\frac{1}{p}}\lesssim_\epsilon (\frac1\delta)^{\epsilon+(\frac12-\frac38)[\frac{m}{2}(1-\alpha_2)+\frac{3m}{2}(1-\alpha_2)(1-\beta_2)]}\times
$$
$$A_p(\frac{3m}{2}, B^{m\cdot\frac32}, \frac{3m}{2})^{(1-\alpha_2)(1-\beta_2)}D_{\frac{5p}{9}}(\frac{3m}4, B^{m\cdot\frac32})^{\beta_2(1-\alpha_2)}D_p(\frac{m}2,B^{m\cdot\frac32})^{\alpha_2}.$$
\end{lem}
\begin{proof}
First, apply Lemma \ref{main1} with $l=2$
\begin{equation}
\label{abc23}
\Big( \frac{1}{|\mc{B}_m(B^{m\cdot\frac32})|} \sum_{B^m\in \mc{B}_m(B^{m\cdot\frac32})} D_{\frac{5p}{9}}(\frac{m}2, B^m)^p \Big)^{\frac{1}{p}}\lesssim_\epsilon\delta^{-\epsilon}D_{\frac{5p}{9}}(\frac{m}2, B^{m\cdot \frac32}).
\end{equation}
Second, use H\"older's inequality,
\beq\label{3004e2.46}
D_{\frac{5p}{9}}(\frac{m}2, B^{m\cdot\frac32}) \lesim D_{8}(\frac{m}2, B^{m\cdot\frac32})^{1-\alpha_2} D_p(\frac{m}2, B^{m\cdot\frac32})^{\alpha_2}.
\endeq
Third, in order to process the term $D_{8}(\frac{m}2, B^{m\cdot\frac32})$, we invoke decoupling \eqref{0626e1.10} for the lower dimensional manifold $\mc{S}_{2,2}$, with $q=\frac83$. Parabolic rescaling (Lemma \ref{abc18}) shows that for each ball $\Delta^{m\cdot\frac32}\subset \R^5$ with radius $\delta^{-m\cdot\frac32}$ we have
\beq
\label{abc19}
\|E_{R_{i, \frac{m}2}}^{(2,2)}g\|_{L^8_{\#}(w_{\Delta^{m\cdot\frac32}})} \lesim_\epsilon \delta^{-\epsilon-\frac{m}2(\frac{1}{2}-\frac{3}{8})} \Big(\sum_{R_{i, \frac{3m}{4}}\subset R_{i, \frac{m}2}} \|E^{(2,2)}_{R_{i, \frac{3m}{4}}}g\|^{\frac{8}{3}}_{L^8_{\#}(w_{\Delta^{m\cdot\frac32}})} \Big)^{\frac{3}{8}}.
\endeq

We combine \eqref{abc19} with  (a weighted version of) the dimension reduction Lemma \ref{abc8} for $\mc{M}_1=\mc{S}_{2,2}$ and $\mc{M}_2=\mc{S}_{2,3}$, to arrive at the inequality (recall $E=E^{(2,3)}$)
\beq\label{0720e3.35h}
\|E_{R_{i, \frac{m}2}}g\|_{L^8_{\#}(w_{B^{m\cdot\frac32}})} \lesim_\epsilon \delta^{-\epsilon-\frac{m}2(\frac{1}{2}-\frac{3}{8})} \Big(\sum_{R_{i, \frac{3m}{4}}\subset R_{i, \frac{m}2}} \|E_{R_{i, \frac{3m}{4}}}g\|^{\frac{8}{3}}_{L^8_{\#}(w_{B^{m\cdot\frac32}})} \Big)^{\frac{3}{8}}.
\endeq
Note that this is an essentially sharp inequality, for the following reason. The piece of the manifold $\mc{S}_{2,3}$ above the square $R_{i, \frac{m}2}$ will have a ``twist" (the effect of the third order terms $t^3,t^2s,ts^2,s^3$) of order only $O(\delta^{\frac{3m}2})$. To explain this, let us consider the square at the origin,  $R_{i,\frac{m}2}=[0,\delta^{\frac{m}2}]^2$. The twist in this case is the maximum amount by which   $\mc{S}_{2,3}$ differs from the purely quadratic manifold
$$\{(t,s,t^2,ts,s^2,0,0,0,0):\;0\le t,s\le \delta^{\frac{m}{2}}\}.$$
Of course, this  is $O(\delta^{\frac{3m}2})$, as claimed. When considering the extension operators on balls $B^{m\cdot\frac32}$ of radius $\delta^{-m\cdot\frac32}$, this observation combined with the uncertainty principle shows that the piece of the manifold $\mc{S}_{2,3}$ above the square $R_{i, \frac{m}2}$ is indistinguishable from (an affine copy of) $\mc{S}_{2,2}$. Thus, there is no twist to be exploited and the best we can do is use the decoupling theory of the quadratic manifold $\mc{S}_{2,2}$.

Now, \eqref{0720e3.35h} has the following immediate implication
\beq\label{abc29}
D_8(\frac{m}2, B^{m\cdot\frac32})\lesssim_\epsilon \delta^{-\epsilon-\frac{m}2(\frac{1}{2}-\frac{3}{8}) } D_8(\frac{3m}4, B^{m\cdot\frac32}).
\endeq
Using \eqref{abc29}, \eqref{3004e2.46} can be further dominated by
\beq\label{3004e2.48}
\delta^{-\epsilon-\frac{m}2(\frac{1}{2}-\frac{3}{8}) (1-\alpha_2)} D_8(\frac{3m}4, B^{m\cdot\frac32})^{ 1-\alpha_2} D_p(\frac{m}2, B^{m\cdot\frac32})^{ \alpha_2}.
\endeq
Another application of H\"older
$$D_8(\frac{3m}4, B^{m\cdot\frac32})\lesim D_2(\frac{3m}4, B^{m\cdot\frac32})^{1-\beta_2}D_{\frac{5p}{9}}(\frac{3m}4, B^{m\cdot\frac32})^{\beta_2}$$
leads to the new bound for \eqref{3004e2.48}
\begin{equation}
\label{abc26}\delta^{-\epsilon-\frac{m}2(\frac{1}{2}-\frac{3}{8}) (1-\alpha_2)} D_p(\frac{m}2, B^{m\cdot\frac32})^{ \alpha_2}D_2(\frac{3m}4, B^{m\cdot\frac32})^{(1-\alpha_2)(1-\beta_2)}D_{\frac{5p}{9}}(\frac{3m}4, B^{m\cdot\frac32})^{(1-\alpha_2)\beta_2}.
\end{equation}

We leave the terms $D_p(\frac{m}2, B^{m\cdot\frac32})$ and $D_{\frac{5p}{9}}(\frac{3m}4, B^{m\cdot\frac32})$ as they are. However, we need to further process the term $D_2(\frac{3m}{4}, B^{m\cdot\frac32})$.
We first invoke $L^2$ orthogonality to pass to  the smallest frequency scales allowed by the uncertainty principle. Then we use H\"older to write
$$\|E_{R_{i, \frac{3m}4}} g\|_{L^2_{\#}(w_{B^{m\cdot\frac32}})}\lesssim (\sum_{R_{i, \frac{3m}{2}}\subset R_{i,\frac{3m}{4}}} \|E_{R_{i, \frac{3m}2}} g\|_{L^2_{\#}(w_{B^{m\cdot\frac32}})}^2)^{1/2}$$
$$
\le \delta^{-\frac{3m}2(\frac12-\frac38)}(\sum_{R_{i, \frac{3m}{2}}\subset R_{i,\frac{3m}{4}}}\|E_{R_{i, \frac{3m}2}} g\|_{L^2_{\#}(w_{B^{m\cdot\frac32}})}^{\frac83})^{\frac38}.
$$
This of course implies that
\beq
\label{abc28}
D_2(\frac{3m}{4}, B^{m\cdot\frac32})\lesim \delta^{-\frac{3m}2(\frac12-\frac38)}D_2(\frac{3m}{2}, B^{m\cdot\frac32}).
\endeq
Plugging this bound in \eqref{abc26} finishes the proof of the lemma.
\end{proof}

{\bf The overture of the $r-$iteration}

In this step, we will start with
\beq
A_p(1, B^3, 1)=\Big( \frac{1}{|\mc{B}_1(B^3)|} \sum_{B^1\in \mc{B}_1(B^3)} D_2(1, B^1)^p \Big)^{1/p}.
\endeq
Our goal is to connect $A_p(1, B^3, 1)$ with quantities of the form $D_t(q, B^3)$ and $A_p(q, B^3, q)$ for some $q> 1$ and $t=\frac{2p}{9}$ or $\frac{5p}{9}$. The two indices $\frac{2p}{9}$ and $\frac{5p}{9}$ are exactly those from Lemma \ref{main1}, for $l=1$ and $l=2$, respectively. The fact that $q> 1$ means that we will pass from  squares of side length $\delta$ to squares of smaller side length $\delta^q$.

First, by H\"older's inequality, since $p>9$
\beq\label{3004e2.41}
\begin{gathered}
\Big( \frac{1}{|\mc{B}_1(B^3)|} \sum_{B^1\in \mc{B}_1(B^3)} D_2(1, B^1)^p \Big)^{1/p}
\lesim \Big( \frac{1}{|\mc{B}_1(B^3)|} \sum_{B^1\in \mc{B}_1(B^3)} D_{\frac{2p}{9}}(1, B^1)^p \Big)^{1/p}.
\end{gathered}
\endeq

Now, we will perform the first ball inflation. Applying the $l=1$ case of  Lemma \ref{main1} to the right hand side of \eqref{3004e2.41}, we obtain
\beq\label{3004e2.42}
\Big( \frac{1}{|\mc{B}_1(B^3)|} \sum_{B^1\in \mc{B}_1(B^3)} D_{\frac{2p}{9}}(1, B^1)^p \Big)^{1/p}\lesssim_\epsilon\delta^{-\epsilon}\Big( \frac{1}{|\mc{B}_2(B^3)|} \sum_{B^2\in \mc{B}_2(B^3)} D_{\frac{2p}{9}}(1, B^2)^p \Big)^{1/p}.
\endeq
Next, we aim  at performing a second ball inflation, according to Lemma \ref{main1} with $l=2$.
By H\"older's inequality, the right hand side of \eqref{3004e2.42} can be dominated by
\beq\label{3004e2.44}
\Big( \frac{1}{|\mc{B}_2(B^3)|} \sum_{B^2\in \mc{B}_2(B^3)} D_{\frac{5p}{9}}(1, B^2)^p \Big)^{\frac{\alpha_1}{p}} \Big( \frac{1}{|\mc{B}_2(B^3)|} \sum_{B^2\in \mc{B}_2(B^3)} D_{2}(1, B^2)^p \Big)^{\frac{1-\alpha_1}{p}} .
\endeq
The motivation for splitting  $L^{2p/9}$ into $L^2$ and $L^{5p/9}$ is twofold. On the one hand, for the  $L^2$ term   we can use an orthogonality argument to perform a further decoupling, more precisely  to pass from  squares of side length $\delta$ to squares of side length $\delta^2$. Combining with H\"older leads to
$$\|E_{R_{i, 1}} g\|^{\frac{8}{3}}_{L^2_{\#}(w_{B^2})}\lesssim (\sum_{R_{i, 2}\subset R_{i,1}} \|E_{R_{i, 2}} g\|_{L^2_{\#}(w_{B^2})}^2)^{1/2}$$
$$
\le \delta^{-2(\frac12-\frac38)}(\sum_{R_{i, 2}\subset R_{i,1}} \|E_{R_{i, 2}} g\|_{L^2_{\#}(w_{B^2})}^{\frac83})^{\frac38}.
$$
This in turn has the following immediate consequence
$$
D_2(1,B^2) \lesssim \delta^{-2(\frac12-\frac38)}D_2(2,B^2)$$
and thus
\begin{equation}
\label{abc20}
\Big( \frac{1}{|\mc{B}_2(B^3)|} \sum_{B^2\in \mc{B}_2(B^3)} D_{2}(1, B^2)^p \Big)^{\frac{1}{p}}\lesssim \delta^{-2(\frac12-\frac38)}A_p(2,B^3,2).
\end{equation}
On the other hand, for the   $L^{\frac{5p}{9}}$ term in \eqref{3004e2.44} we can apply Lemma \ref{abc22} with $m=2$.
\bigskip

Putting these observations  together, we obtain
\beq\label{3004e2.51}
\begin{split}
& A_p(1, B^3, 1)  \lesim_\epsilon (\frac{1}{\delta})^{\epsilon+2(\frac{1}{2}-\frac{3}{8})(1-\alpha_1)+(\frac{1}{2}-\frac{3}{8})\alpha_1 (1-\alpha_2)+3 (\frac{1}{2}-\frac{3}{8}) \alpha_1 (1-\alpha_2) (1-\beta_2)}\times \\
		& A_p(2, B^3, 2)^{1-\alpha_1} A_p(3, B^3, 3)^{\alpha_1 (1-\alpha_2)(1-\beta_2)}D_{\frac{5p}{9}}(\frac{3}{2}, B^3)^{\alpha_1 (1-\alpha_2)\beta_2} D_p(1, B^3)^{\alpha_1 \alpha_2}.
\end{split}
\endeq
This finishes the overture of the $r-$iteration.
\bigskip

Next, we will repeat the $l=2$ ball-inflation for the term $D_{\frac{5p}{9}}(\frac{3}{2}, B^3)$, which will represent the generic step of the $r-$iteration. Note that so far we have used the $l=2$ ball inflation to increase the radius of the balls from $\delta^{-2}$ to $\delta^{-3}$, that is, the exponent of $\delta^{-1}$ has grown by a multiplicative factor of $3/2$.  In the second step described below, the radius will similarly grow from $\delta^{-3}$ to $\delta^{-9/2}$. Each step of the iteration will increase the exponent by the same factor $3/2$.
\medskip

{\bf The first step of the $r-$iteration}

We will average \eqref{3004e2.51} raised to the power $p$ over a finitely overlapping cover of $B^{\frac{9}{2}}$ using balls $B^3$. Apart from the $\delta$ term, there are four main terms in \eqref{3004e2.51} and their exponents add up to 1
$${1-\alpha_1}+{\alpha_1 (1-\alpha_2)(1-\beta_2)}+{\alpha_1 (1-\alpha_2)\beta_2}+{\alpha_1 \alpha_2}=1$$
This allows us to estimate the sum over the balls $B^3$ using H\"older. Note that the $p$th powers of the terms $A_p$ sum up rather naturally. For the sum of the $p$th powers of the terms $D_p$ we use  Lemma \ref{abc31}. Finally, the $p$th powers of the terms $D_{\frac{5p}{9}}$ are estimated using Lemma \ref{abc22}, this time with $m=3$. We get

\beq
\begin{split}
& A_p(1, B^{\frac{9}{2}}, 1)  \lesim_\epsilon (\frac{1}{\delta})^{\epsilon+2(\frac{1}{2}-\frac{3}{8}) (1-\alpha_1)} \underbrace{(\frac{1}{\delta})^{(\frac{1}{2}-\frac{3}{8}) \alpha_1 (1-\alpha_2)} }_{l^{\83} L^8 \text{ decoupling }} \underbrace{(\frac{1}{\delta})^{\frac{3}{2}(\frac{1}{2}-\frac{3}{8}) \alpha_1 (1-\alpha_2)^2 \beta_2} }_{l^{\83} L^8 \text{ decoupling }}\times\\
& \underbrace{(\frac{1}{\delta})^{3(\frac{1}{2}-\frac{3}{8}) \alpha_1 (1-\alpha_2)(1-\beta_2)}}_{L^2 \text{ orthogonality }} \underbrace{(\frac{1}{\delta})^{\frac{9}{2}(\frac{1}{2}-\frac{3}{8}) \alpha_1 (1-\alpha_2)^2 \beta_2 (1-\beta_2)} }_{L^2 \text{ orthogonality }} \times\\
& A_p(2, B^{\frac{9}{2}}, 2)^{1-\alpha_1} A_p(3, B^{\frac{9}{2}}, 3)^{\alpha_1 (1-\alpha_2)(1-\beta_2)} D_p(1, B^{\frac{9}{2}})^{\alpha_1 \alpha_2}\times\\
& A_p(\frac{9}{2}, B^{\frac{9}{2}}, \frac{9}{2})^{\alpha_1 (1-\alpha_2)^2 \beta_2 (1-\beta_2)} D_p(\frac{3}{2}, B^{\frac{9}{2}})^{\alpha_1 \alpha_2 (1-\alpha_2) \beta_2} \times D_{\frac{5p}{9}}(\frac{9}{4}, B^{\frac{9}{2}})^{\alpha_1 (1-\alpha_2)^2 \beta_2^2}.
\end{split}
\endeq
This finishes the first step of the ball-inflation argument.\\

{\bf The end result of the $r-$iteration}. We repeat the procedure described in the first step $r-1$ times, each time increasing the size of the ball by a factor of $\frac32$. We obtain that for all balls $B$ of radius $\delta^{-2\cdot (\frac{3}{2})^r}$
\beq\label{3004e2.55k}
\begin{split}
& A_p(1, B, 1) \lesim_{\epsilon,r} (\frac{1}{\delta})^{\epsilon+2(\frac{1}{2}-\frac{3}{8})(1-\alpha_1)} \underbrace{\prod_{i=1}^r (\frac{1}{\delta})^{2(\frac{3}{2})^i (\frac{1}{2}-\frac{3}{8}) \alpha_1 (1-\alpha_2) (1-\beta_2) [(1-\alpha_2)\beta_2]^{i-1}}}_{L^2 \text{ orthogonality }}\times \\
& \underbrace{\prod_{i=0}^{r-1} (\frac{1}{\delta})^{(\frac{3}{2})^i (\frac{1}{2}-\frac{3}{8}) \alpha_1 (1-\alpha_2) [(1-\alpha_2)\beta_2]^{i}}}_{l^{\83} L^8 \text{ decoupling }} \times A_p(2, B, 2)^{1-\alpha_1}  D_{\frac{5p}{9}}\Big((\frac{3}{2})^r, B \Big)^{\alpha_1 [(1-\alpha_2)\beta_2]^r}\\
& \left(\prod_{i=1}^r A_p(2 (\frac{3}{2})^i, B, 2 (\frac{3}{2})^i)^{\alpha_1 (1-\alpha_2)(1-\beta_2)[(1-\alpha_2)\beta_2]^{i-1}}\right)\left(\prod_{i=0}^{r-1}D_p((\frac{3}{2})^i, B)^{\alpha_1 \alpha_2 [(1-\alpha_2)\beta_2]^i} \right).
\end{split}
\endeq
Define
\beq\label{0720e3.48h}
\begin{split}
& \gamma_0=1-\alpha_1; \gamma_i=\alpha_1 (1-\alpha_2)(1-\beta_2)[(1-\alpha_2)\beta_2]^{i-1}, \text{ for } 1\le i\le r;\\
& b_i=2\cdot (\frac{3}{2})^i, \text{ for } 0\le i\le r;\\
& \tau_r=\alpha_1 [(1-\alpha_2)\beta_2]^r; \tau_i=\alpha_1 \alpha_2 [(1-\alpha_2)\beta_2]^i, \text{ for }0\le i\le r-1;\\
& w_i= \frac{1-\alpha_2}{2\alpha_2}\tau_i, \text{ for } 0\le i\le r-1.
\end{split}
\endeq
We can write using H\"older
$$D_{\frac{5p}{9}}\Big((\frac{3}{2})^r, B \Big)\lesssim D_{p}\Big((\frac{3}{2})^r, B \Big).$$
With these,  the estimate \eqref{3004e2.55k} becomes
\beq\label{3004e2.57p}
\begin{split}
A_p(1, B, 1)& \lesim_{r,\epsilon} \Big(\prod_{i=0}^r (\frac{1}{\delta})^{\epsilon+(\frac{1}{2}-\frac{3}{8})b_i \gamma_i}\Big)\Big( \prod_{i=0}^{r-1} (\frac{1}{\delta})^{ (\frac{1}{2}-\frac{3}{8}) b_i w_i} \Big)\times  \\ &\Big(\prod_{i=0}^r A_p(b_i, B, b_i)^{\gamma_i} \Big)\Big( \prod_{i=0}^{r}D_p(\frac{b_i}{2}, B)^{\tau_i} \Big)
\end{split}
\endeq

By invoking Lemma \ref{abc31} and H\"older's inequality, here $B$ can in fact be any ball of radius bigger than $\delta^{-2 (\frac{3}{2})^r}$. By renaming the variable $\delta$, we arrive at the following key result.
\begin{prop}
Let $p\ge \frac{72}5.$ Let $u>0$ be a small number such that $u\cdot (\frac{3}{2})^r\le 1$. Then for each ball $B$ of radius $\delta^{-3}$, we have
\beq\label{3004e2.60p}
\begin{split}
A_p(u, B, u)& \lesim_{r,\epsilon} \Big(\prod_{i=0}^r (\frac{1}{\delta})^{\epsilon+u(\frac{1}{2}-\frac{3}{8})b_i \gamma_i}\Big)\Big( \prod_{i=0}^{r-1} (\frac{1}{\delta})^{u(\frac{1}{2}-\frac{3}{8}) b_i w_i} \Big)\times \\
& \Big(\prod_{i=0}^r A_p(ub_i, B, ub_i)^{\gamma_i} \Big)\Big( \prod_{i=0}^{r} D_p(\frac{u b_i}{2}, B)^{\tau_i} \Big).
\end{split}
\endeq
\end{prop}

Recall that in the definition of the quantity $D_p$  we have used an $l^{\frac{8}{3}}$ summation. However as we are eventually aiming at proving an  $l^p L^p$ decoupling inequality  (for $p=20$), we also need to introduce the following quantity:
\beq
\tilde{D}_p (q, B^r):=\Big( \prod_{i=1}^{\Lambda K} \sum_{R_{i, q}\subset R_i}\|E_{R_{i, q}} g\|^{p}_{L^p_{\#}(w_{B^r})} \Big)^{\frac{1}{p\Lambda K}}.
\endeq
By invoking H\"older's inequality, we get for $p\ge \frac83$
\beq
\label{abc36}
{D}_p (q, B)\le \delta^{-2q(\frac38-\frac1p)}\tilde{D}_p (q, B).
\endeq
Using this, we can rewrite \eqref{3004e2.60p} as follows
\beq\label{0709e2.60k}
\begin{split}
A_p(u, B, u)& \lesim_{\epsilon,r} \Big(\prod_{i=0}^r (\frac{1}{\delta})^{\epsilon+u(\frac{1}{2}-\frac{3}{8})b_i \gamma_i}\Big)\Big( \prod_{i=0}^{r-1} (\frac{1}{\delta})^{u(\frac{1}{2}-\frac{3}{8}) b_i w_i} \Big)\Big( \prod_{i=0}^{r} (\frac{1}{\delta})^{u(\frac{3}{8}-\frac{1}{p}) b_i \tau_i} \Big)\times \\
& \Big(\prod_{i=0}^r A_p(ub_i, B, ub_i)^{\gamma_i} \Big)\Big( \prod_{i=0}^{r} \tilde{D}_p(\frac{u b_i}{2}, B)^{\tau_i} \Big).
\end{split}
\endeq

This inequality  is ready for the $M-$iteration.
\medskip

{\bf The $M-$iteration}
\medskip

To iterate, we will dominate each  $A_p(ub_i, B, ub_i)$ again by using \eqref{0709e2.60k}. To enable such an iteration, we need to choose $u$ to be even smaller. Let $M$ be a large integer. Choose $u$ such that
\beq
[2(\frac{3}{2})^r]^M u\le 2.
\endeq
 This allows us to iterate \eqref{0709e2.60k} $M$ times. When iterating \eqref{0709e2.60k}, we always need to carry the original $\tilde{D}_p$-terms. However, there is some simplification that one can make. We bound the power of $\frac{1}{\delta}$ by
\beq
\left( \sum_{i=0}^{\infty} u (\frac{1}{2}-\frac{3}{8}) b_i \gamma_i\right)+\left( \sum_{i=0}^{\infty} u (\frac{1}{2}-\frac{3}{8}) b_i w_i \right)+\left( \sum_{i=0}^{\infty} u (\frac{3}{8}-\frac{1}{p}) b_i \tau_i \right).
\endeq
By a direct calculation,
\beq
\sum_{j=0}^{\infty}b_j \gamma_j=\frac{6(13p-216)}{5p^2-94p+144},
\endeq
\beq
\sum_{j=0}^{\infty} b_j w_j=\frac{32(p-9)}{5p^2-94p+144},
\endeq
and
\beq
\sum_{j=0}^{\infty} b_j \tau_j=\frac{2 (648 - 117 p + 5 p^2)}{144 - 94 p + 5 p^2}.
\endeq
If we define
\beq
\lambda_0:= (\frac{1}{2}-\frac{3}{8}) \left( \frac{6(13p-216)}{5p^2-94p+144}+\frac{32(p-9)}{5p^2-94p+144}\right)+(\frac{3}{8}-\frac{1}{p})\left(\frac{2 (648 - 117 p + 5 p^2)}{144 - 94 p + 5 p^2}\right),
\endeq
then \eqref{0709e2.60k} can be rewritten as follows
\beq\label{0405e2.45}
\begin{split}
A_p(u, B, u) \lesim_{r,\epsilon} \delta^{-\epsilon-u \lambda_0}  \Big(\prod_{i=0}^r A_p(ub_i, B, ub_i)^{\gamma_i} \Big)\Big( \prod_{i=0}^{r} \tilde{D}_p(\frac{u b_i}{2}, B)^{\tau_i} \Big).
\end{split}
\endeq
Now we iterate the above estimate  $M$ times, and obtain
\beq\label{0405e2.46}
\begin{split}
A_p(u, B, u)& \lesim_{\epsilon,r,M} \delta^{-u \lambda_0-\epsilon} \left(\prod_{j_1=0}^{r} \delta^{-u \lambda_0 b_{j_1}\gamma_{j_1}} \right)\times  \\ & \ldots\\& \left(\prod_{j_1=0}^{r}\prod_{j_2=0}^{r}... \prod_{j_{M-1}=0}^{r} \delta^{-u \lambda_0 b_{j_1}b_{j_2}... b_{j_{M-1}}\gamma_{j_1}\gamma_{j_2}... \gamma_{j_{M-1}}} \right)\times \\ &
\left( \prod_{j_1=0}^{r} \tilde{D}_p(\frac{u}{2}\cdot  b_{j_1}, B)^{\tau_{j_1}}\right) \left( \prod_{j_1=0}^{r} \prod_{j_2=0}^{r} \tilde{D}_p(\frac{u}{2} \cdot  b_{j_1}b_{j_2}, B)^{\tau_{j_1}\gamma_{j_2}}\right)\times \\
				& \dots \\
				& \left( \prod_{j_1=0}^{r} \prod_{j_2=0}^{r}\dots \prod_{j_M=0}^{r} \tilde{D}_p(\frac{u}{2}\cdot b_{j_1}b_{j_2}... b_{j_M}, B)^{\tau_{j_1}\gamma_{j_2}... \gamma_{j_M}}\right)\times \\
				& \left( \prod_{j_1=0}^{r} \prod_{j_2=0}^{r}\dots \prod_{j_M=0}^{r} A_p(u\cdot b_{j_1}b_{j_2}... b_{j_M}, B, u\cdot b_{j_1}b_{j_2}... b_{j_M})^{\gamma_{j_1}\gamma_{j_2}... \gamma_{j_M}}\right).
\end{split}
\endeq

We start to process the long product \eqref{0405e2.46}. We will divide it into three steps. In the first step, we collect all the powers of $\frac{1}{\delta}$. In the second, we use a rescaling argument to handle all the $\tilde{D}_p$-terms. In the last step, we deal with the remaining $A_p$-terms.\\

\noindent {\it Collecting the powers of $\frac{1}{\delta}$.} We obtain
\beq
\begin{split}
& u\lambda_0 + u\lambda_0 (\sum_{j=0}^r b_j \gamma_j)+\dots + u \lambda_0 (\sum_{j=0}^r b_j \gamma_j)^{M-1}\\
& = u \lambda_0\cdot  \frac{1-(\sum_{j=0}^r b_j \gamma_j)^M}{1-(\sum_{j=0}^r b_j \gamma_j)}.
\end{split}
\endeq

\noindent {\it The contribution from the $\tilde{D}_p$-terms.} By  parabolic rescaling (Lemma \ref{abc18}), the product of all these $\tilde{D}_p$-terms can be controlled by
\beq
\begin{split}
&\left( \prod_{j_1=0}^r V_p(\delta^{1-\frac{u}{2}b_{j_1}})^{\tau_{j_1}} \tilde{D}_p(1, B)^{\tau_{j_1}}\right) \times \left(\prod_{j_1=0}^r\prod_{j_2=0}^r V_p(\delta^{1-\frac{u}{2}b_{j_1}b_{j_2}})^{\tau_{j_1}\gamma_{j_2}} \tilde{D}_p(1, B)^{\tau_{j_1}\gamma_{j_2}} \right)\\
&\times \dots \times \left(\prod_{j_1=0}^r\prod_{j_2=0}^r\dots \prod_{j_M=0}^r V_p(\delta^{1-\frac{u}{2}b_{j_1}b_{j_2}\dots b_{j_M}})^{\tau_{j_1}\gamma_{j_2}\dots \gamma_{j_M}} \tilde{D}_p(1, B)^{\tau_{j_1}\gamma_{j_2}\dots \gamma_{j_M}} \right)\\
& \lesim \left( \prod_{j_1=0}^r V_p(\delta^{1-\frac{u}{2}b_{j_1}})^{\tau_{j_1}} \right) \times \left(\prod_{j_1=0}^r\prod_{j_2=0}^r V_p(\delta^{1-\frac{u}{2}b_{j_1}b_{j_2}})^{\tau_{j_1}\gamma_{j_2}} \right) \times \dots\\
&  \times \left(\prod_{j_1=0}^r\prod_{j_2=0}^r\dots \prod_{j_M=0}^r V_p(\delta^{1-\frac{u}{2}b_{j_1}b_{j_2}\dots b_{j_M}})^{\tau_{j_1}\gamma_{j_2}\dots \gamma_{j_M}}  \right) \Big(\tilde{D}_p(1, B)\Big)^{1-(\sum_{j=0}^r \gamma_j)^M}
\end{split}
\endeq

\noindent {\it The contribution from the $A_p$-term.} By invoking \eqref{abc35} and \eqref{abc36} this term can be bounded by
\beq
\prod_{j_1=0}^r \dots \prod_{j_M=0}^r (\frac{1}{\delta})^{2ub_{j_1}... b_{j_M} \gamma_{j_1}... \gamma_{j_M}}\left[  \tilde{D}_p(b_{j_1}\dots b_{j_M}u, B) \right]^{\gamma_{j_1}\dots \gamma_{j_M}}.
\endeq
To control the $\tilde{D}_p$ term, we again invoke the parabolic rescaling, and bound the last expression by
\beq
(\frac{1}{\delta})^{2u(\sum_{j=0}^r b_j \gamma_j)^M}\prod_{j_1=0}^r \dots \prod_{j_M=0}^r \Big( V_p(\delta^{1-ub_{j_1}\dots b_{j_M}}) \Big)^{\gamma_{j_1}\dots \gamma_{j_M}} \Big( \tilde{D}_{p}(1, B) \Big)^{\gamma_{j_1}\dots \gamma_{j_M}}.
\endeq
We summarize our findings so far as follows, recalling that the implicit constant also depends on $K$
\begin{prop}\label{0712prop1.10}
For each $p>\frac{72}{5}$ , for each ball $B$ of radius $\delta^{-3}$, and for each sufficiently small $u$, we have
\beq\label{3004e2.77}
\begin{split}
& A_p(u, B, u)\lesim_{\epsilon,r,M,K} (\frac{1}{\delta})^{\epsilon+u \lambda_0\cdot  \frac{1-(\sum_{j=0}^r b_j \gamma_j)^M}{1-(\sum_{j=0}^r b_j \gamma_j)}+2u(\sum_{j=0}^r b_j \gamma_j)^M} \tilde{D}_p(1, B)\times  \left( \prod_{j_1=0}^r V_p(\delta^{1-\frac{u}{2}b_{j_1}})^{\tau_{j_1}} \right) \times \\
 & \left(\prod_{j_i=0}^r\prod_{j_2=0}^r V_p(\delta^{1-\frac{u}{2}b_{j_1}b_{j_2}})^{\tau_{j_1}\gamma_{j_2}} \right) \times \dots \times \left(\prod_{j_1=0}^r\prod_{j_2=0}^r\dots \prod_{j_M=0}^r V_p(\delta^{1-\frac{u}{2}b_{j_1}b_{j_2}\dots b_{j_M}})^{\tau_{j_1}\gamma_{j_2}\dots \gamma_{j_M}}  \right)\\
 &  \left(\prod_{j_1=0}^r \dots \prod_{j_M=0}^r  \Big( V_p(\delta^{1-ub_{j_1}\dots b_{j_M}} )\Big)^{\gamma_{j_1}\dots \gamma_{j_M}} \right) .
\end{split}
\endeq
\end{prop}

\noindent {\bf The final step of the proof.} Now we come to the final step of the proof for the desired estimate \eqref{3004e2.3} at the critical exponent $p=20$. We will combine  Theorem \ref{abc37} with  Proposition \ref{0712prop1.10}.  For $p>\frac{72}{5}$ let $\eta_p$ be the unique number such that
\beq\label{0705g2.52}
\lim_{\delta\to 0} \frac{V_{p}(\delta)}{\delta^{-(\eta_p+\mu)}}=0, \text{ for each }\mu>0,
\endeq
and
\beq
\limsup_{\delta\to 0} \frac{V_{p}(\delta)}{\delta^{-(\eta_p-\mu)}}=\infty, \text{ for each }\mu>0.
\endeq
Let $B$ have radius $\delta^{-3}$. We substitute the bound $V_{p}(\delta)\lesim_{\mu} \delta^{-(\eta_p+\mu)} $ into the right hand side of \eqref{3004e2.77}, and obtain
\beq
A_p(u, B, u)\lesim_{r,M,K,\mu} \delta^{-\eta_{p, \mu, u, r, M}} \tilde{D}_p(1, B),
\endeq
where
\beq\label{0712e1.60}
\begin{split}
\eta_{p, \mu, u,r,M}& =u \lambda_0\cdot  \frac{1-(\sum_{j=0}^r b_j \gamma_j)^M}{1-(\sum_{j=0}^r b_j \gamma_j)} +2u(\sum_{j=0}^r b_j \gamma_j)^M\\
	& + (\mu+\eta_p) \left[ 1-u\cdot (\sum_{j=0}^r b_j \gamma_j)^M -\frac{u}{2}(\sum_{j=0}^r b_j \tau_j) \frac{1-(\sum_{j=0}^r b_j \gamma_j)^M}{1-(\sum_{j=0}^r b_j \gamma_j)} \right].
\end{split}
\endeq
By Cauchy--Schwarz,
\beq
\begin{split}
& \left\| (\prod_{i=1}^{\Lambda K} E_{R_i} g)^{\frac{1}{\Lambda K}} \right\|_{L_{\#}^p(w_B)} \le \delta^{-\frac{5u}{4}} \left\| (\prod_{i=1}^{\Lambda K} \sum_{R_{i, u}\subset R_i} |E_{R_{i, u}}g|^{\frac{8}{3}})^{\frac{3}{8{\Lambda K}}} \right\|_{L_{\#}^p(w_B)}\\
& \lesim \delta^{-\frac{5u}{4}} \left( \frac{1}{|\mathcal{B}_u(B)|} \sum_{B^u\in \mathcal{B}_u(B)}\left\| (\prod_{i=1}^{\Lambda K} \sum_{R_{i, u}\subset R_i} |E_{R_{i, u}}g|^{\frac{8}{3}})^{\frac{3}{8{\Lambda K}}} \right\|_{L_{\#}^p(w_{B^u})}^p \right)^{\frac{1}{p}} .
\end{split}
\endeq
By H\"older and Minkowski, this can be further bounded by
\beq
 \delta^{-\frac{5u}{4}} \left( \frac{1}{|\mathcal{B}_u(B)|} \sum_{B^u\in \mathcal{B}_u(B)} D_p(u, B^u)^p \right)^{\frac{1}{p}}= \delta^{-\frac{5u}{4}} A_p(u, B, u).
\endeq
So far we have obtained
\beq
\| (\prod_{i=1}^M E_{R_i} g)^{\frac{1}{M}} \|_{L_{\#}^p(w_B)}  \lesim_{r,M,K,\mu} \delta^{-\frac{5u}{4}-\eta_{p, \mu, u,r,M}} \tilde{D}_p(1, B).
\endeq
We recall that both sides depend on $g$ and $R_i$.
By taking the supremum over $g$ and $R_i$ (with fixed $K$) in the above estimate, we obtain
\beq
\label{abc43}
V_{p}(\delta, K)\lesim_{r,M,K,\mu} \delta^{-\tilde{\eta}_{p, \mu, u,r,M}},
\endeq
where
\beq
\tilde{\eta}_{p, \mu, u,r,M}:=\eta_{p, \mu, u,r,M}+\frac{5u}{4}.
\endeq
We move $\eta_p$ from the right hand side of the expression \eqref{0712e1.60} to the left hand  side, and then divide both sides by $u$ to obtain
\beq\label{0712e1.65}
\begin{split}
\frac{1}{u}(\tilde{\eta}_{p, \mu, u,r,M}-\eta_p)& =\frac{5}{4}+\frac{\mu}{u}+\lambda_0\cdot  \frac{1-(\sum_{j=0}^r b_j \gamma_j)^M}{1-(\sum_{j=0}^r b_j \gamma_j)}+2(\sum_{j=0}^r b_j \gamma_j)^M\\
		& -(\mu+\eta_p) \left[ (\sum_{j=0}^r b_j \gamma_j)^M +\frac{1}{2}(\sum_{j=0}^r b_j \tau_j) \frac{1-(\sum_{j=0}^r b_j \gamma_j)^M}{1-(\sum_{j=0}^r b_j \gamma_j)} \right].
\end{split}
\endeq
Our goal is to show that
\beq\label{0712f1.66}
\eta_{20}\le 2 (\frac{1}{2}-\frac{1}{20}).
\endeq
We argue by contradiction. Suppose for contradiction that
\beq
\label{abc50}
\eta_{20}> \frac{9}{10}.
\endeq
Then, for sufficiently small $\epsilon_1>0$ this forces
\beq\label{0712f1.67}
\eta_p>\frac{9}{10}, \text{ for each } p\in (20-\epsilon_1, 20).
\endeq
We rewrite the right hand side of \eqref{0712e1.65} as
\beq
\label{abc38}
\underbrace{\Big( \lambda_0-\frac{1}{2}\cdot (\mu+\eta_p)(\sum_{j=0}^r b_j \tau_j) \Big) \frac{1-(\sum_{j=0}^r b_j \gamma_j)^M}{1-(\sum_{j=0}^r b_j \gamma_j)}}_{\text{dominant term}} +\frac{5}{4}+\frac{\mu}{u}+(2-\mu-\eta_p)(\sum_{j=0}^r b_j \gamma_j)^M
\endeq
We have calculated that
$$
\sum_{j=0}^{\infty}b_j \gamma_j=\frac{6(13p-216)}{5p^2-94p+144}.
$$
The two crucial features for this quantity are as follows. First, when $p$ is smaller than (and sufficiently close to) the critical exponent 20,
\beq
\label{abc40}
\sum_{j=0}^{\infty}b_j \gamma_j>1.
\endeq
Second,
\beq
\label{abc42}
\lim_{p\to 20}\sum_{j=0}^{\infty}b_j \gamma_j=1.
\endeq
In addition to these, by a direct calculation we observe that
\beq
\label{abc41}
\lim_{p\to 20} \frac{\lambda_0}{\frac{1}{2}\cdot (\sum_{j=0}^{\infty} b_j\tau_j)}=\frac{9}{10}.
\endeq
Choose now $p$ close enough to 20,  $r$ and $M$ large enough, and then $\mu$  small enough.
By combining \eqref{0712f1.67}, \eqref{abc40}, \eqref{abc42} and \eqref{abc41} we get that for these values of $p,r,M,\mu$ we have
$$
\eqref{abc38}<0.
$$
Going back to \eqref{0712e1.65}, for these values of $p,\mu,r,M$ we conclude that
\beq\label{0712f1.70}
\tilde{\eta}_{p, \mu, u,r,M}<\eta_p.
\endeq
 Together with \eqref{0712f1.67}, for an exponent $p$ slightly smaller than the critical exponent $20$ and for $K$ large enough, Theorem \ref{abc37} implies that
\beq\label{0712f1.71}
V_p(\delta)\le \Omega_{K, p} \log_K \big(\frac{1}{\delta} \big)\max_{\delta\le \delta'\le 1}(\frac{\delta'}{\delta})^{2(\frac{1}{2}-\frac{1}{p})} V_p(\delta', K).
\endeq
We have two possibilities. First, if
\beq
\tilde{\eta}_{p, \mu, u,r,M}< 2(\frac{1}{2}-\frac{1}{p}),
\endeq
then \eqref{0712f1.71} combined with \eqref{abc43} forces $$V_p(\delta)\lesssim_\epsilon(\frac{1}{\delta})^{\epsilon+2(\frac{1}{2}-\frac{1}{p})}.$$ This is a contradiction to \eqref{0712f1.67}.

Second, if
\beq
\tilde{\eta}_{p, \mu, u,r,M}\ge 2(\frac{1}{2}-\frac{1}{p}),
\endeq
then again \eqref{0712f1.71} combined with \eqref{abc43} forces $$V_p(\delta)\lesssim_\epsilon (\frac{1}{\delta})^{\epsilon+\tilde{\eta}_{p, \mu, u,r,M}}.$$ This is a contradiction to \eqref{0712f1.70}. Since both cases lead to a contradiction, it can only be that our original assumption \eqref{abc50} is false. This  finishes the proof of \eqref{0712f1.66}.

\section{Appendix: Some linear algebra}

In this section, we prove Conjecture \ref{conjecture0} for $(d,k,l)=(2,3,2)$. At each point $(r, s)\in [0, 1]^2$, define five vectors
\beq\label{0730e5.1h}
\begin{split}
& \P_1(r, s)=\Phi_r(r, s)=(1, 0, 2r, s, 0, 3r^2, 2rs, s^2, 0)^T\\
& \P_2(r, s)=\P_s(r, s)=(0, 1, 0, r, 2s, 0, r^2, 2rs, 3s^2)^T\\
& \P_3(r, s)=\P_{rr}(r, s)=(0, 0, 2, 0, 0, 6r, 2s, 0, 0)^T\\
& \P_4(r, s)=\P_{rs}(r, s)=(0, 0, 0, 1, 0, 0, 2r, 2s, 0)^T\\
& \P_5(r, s)=\P_{ss}(r, s)=(0, 0, 0, 0, 2, 0, 0, 2r, 6s)^T.
\end{split}
\endeq
Here we use the transpose operation ``T'' to turn all row vectors to column vectors.  Moreover, we define the $9\times 2$ matrix
\beq
\mathcal{M}^{(1)}(r, s)=[\P_1(r, s)^T, \P_2(r, s)^T],
\endeq
and the $9\times 5$ matrix
\beq
\mathcal{M}^{(2)}(r, s)=[\P_1(r, s)^T, \P_2(r, s)^T, \P_3(r, s)^T, \P_4(r, s)^T, \P_5(r, s)^T].
\endeq
Take a linear subspace $V=\langle v_1, ..., v_{dim(V)}\rangle\subset \R^9$. For the sake of convenience, we also assume that $v_i$ is a column vector. Denote
\beq
\mc{M}_V^{(l)}(r, s)=(v_1, v_2, ..., v_{dim(V)})^T \times \mathcal{M}^{(l)}(r, s).
\endeq
We will prove
\begin{prop}
For each $l\in \{1, 2\}$, and each linear subspace $V\subset \R^9$ with dimension $dim(V)$, the matrix $\mc{M}_V^{(l)}$ has at least one minor determinant of order
\beq\label{0730e5.5h}
\Big[ \frac{dim(V)\cdot l(l+3)}{18} \Big]+1,
\endeq
which, viewed as a function of $(r, s)\in [0, 1]^2$, does not vanish identically.
\end{prop}

The case $l=1$ has been verified by Bourgain and Demeter \cite{BD2}. The rest of this section is devoted to the proof of this proposition for the case $l=2$. \\


Before we start the proof, we introduce some more notations. Let $r$ and $s$ be two variables. Define the vector spaces of polynomials
\beq
S_0=[1], S_1=[r, s], S_2=[r^2, rs, s^2], S_3=[r^3, r^2 s, r s^2, s^3].
\endeq
 For $\xi=(a, b)\in \R^2$, let
\beq\label{bg1.1}
\begin{split}
P_{\xi} f(r, s)& =f(\xi)+\partial_r f(\xi)\cdot (r-a)+\partial_s f(\xi)\cdot (s-b)\\
		& +\frac{1}{2}\partial_{rr} f(\xi)\cdot (r-a)^2+\partial_{rs} f(\xi)\cdot (r-a)(s-b)+\frac{1}{2}\partial_{ss} f(\xi)\cdot (s-b)^2
\end{split}
\endeq
be the Taylor expansion of order two of the function $f$ at $\xi$. Hence $P_{\xi}$ is a projection on $S_0\oplus S_1\oplus S_2$.

Denoting $\pi_{1, 2}=\pi_{S_1\oplus S_2}$, we have by \eqref{bg1.1}
\beq\label{bg1.2}
\begin{split}
\pi_{1, 2} P_{\xi} f(r, s)& =(\partial_r f(\xi)-a \partial_{rr} f(\xi)-b \partial_{rs} f(\xi))\cdot r+(\partial_s f(\xi)-a \partial_{rs} f(\xi)-b \partial_{ss} f(\xi))\cdot s\\
			& +\frac{1}{2} \partial_{rr} f(\xi) \cdot r^2+\partial_{rs} f(\xi)\cdot rs+\frac{1}{2} \partial_{ss} f(\xi)\cdot s^2.
\end{split}
\endeq
The action of $\pi_{1, 2}P_{\xi}$ on $S_3$ is given by
\beq\label{bg1.3}
\begin{split}
& \pi_{1, 2} P_{\xi} (r^3)=(-3a^2, 0, 3a, 0, 0),\\
& \pi_{1, 2} P_{\xi} (r^2 s)=(-2ab, -a^2, b, 2a, 0),\\
& \pi_{1, 2} P_{\xi} (rs^2)=(-b^2, -2ab, 0, 2b, a),\\
& \pi_{1, 2} P_{\xi} (s^3)=(0, -3b^2, 0, 0, 3b).
\end{split}
\endeq
Hence if $a\neq 0$ and $b\neq 0$, then $\pi_{1, 2} P_{\xi} (S^3)\subset S_1\oplus S_2$ is the three dimensional space generated by
\beq\label{bg1.4}
\begin{split}
& (a, 0, -1, 0, 0),\\
& (0, b, 0, 0, -1),\\
& (-b, -a, 0, 2, 0).
\end{split}
\endeq
\begin{lem}\label{bglemma1}
Assume $f_1, f_2\in S_3$ linearly independent. Then
\beq
\text{dim} (\pi_{1, 2} P_{\xi} ([f_1, f_2]))=2, \text{ for $\xi$ almost surely} .
\endeq
\end{lem}
\begin{proof}
We argue by contradiction. Assume that $\pi_{1, 2}P_{\xi}f_1$ and $\pi_{1, 2}P_{\xi}f_2$ are linearly dependent for all $\xi$. By \eqref{bg1.2}, this means that
\beq\label{bg2.2}
rank
\begin{pmatrix}
\partial_r f_1, & \partial_s f_1, & \partial_{rr} f_1, & \partial_{rs} f_1, & \partial_{ss} f_1\\
\partial_r f_2, & \partial_s f_2, & \partial_{rr} f_2, & \partial_{rs} f_2, & \partial_{ss} f_2
\end{pmatrix}
=1.
\endeq
Hence
\beq
\det \begin{pmatrix}
\partial_r f_1, & \partial_r f_2\\
\partial_{rr} f_1, & \partial_{rr} f_2
\end{pmatrix}
=0=
\det \begin{pmatrix}
\partial_r f_1, & \partial_r f_2\\
\partial_{rs} f_1, & \partial_{rs} f_2
\end{pmatrix}
\endeq
implying linear dependence of $\partial_r f_1$ and $\partial_r f_2$ by the generalised Wronskian theorem (see for instance \cite{BD000}). Thus we may assume that $f_2=f_1+g(s)$ and since also
\beq
\det \begin{pmatrix}
\partial_r f_1, & \partial_r f_2\\
\partial_{s} f_1, & \partial_{s} f_2
\end{pmatrix}
=
\det \begin{pmatrix}
\partial_r f_1, & \partial_r f_1\\
\partial_{s} f_1, & \partial_{s} f_1+g'
\end{pmatrix}
=\partial_r f_1\cdot g',
\endeq
either $g$ is a constant, hence $g=0$ (contradiction) or $\partial_r f_1=0$. Similarly $\partial_s f_1=0$ so that $f_1$ is constant, which is again a contradiction.
\end{proof}

Denote $S=S_1\oplus S_2\oplus S_3$ and $V$ a subspace of $S$. We need to prove that almost surely in $\xi$,
\beq\label{bg00}
\begin{split}
& dim[(\partial_r f(\xi), \partial_s f(\xi), \partial_{rr} f(\xi), \partial_{rs} f(\xi), \partial_{ss} f(\xi)); f\in V]=dim(\pi_{1, 2} P_{\xi}(V))\\
& \ge  \begin{cases}
               5 \text{ if } dim(V)=8 \\
               4 \text{ if } dim(V)=6\\
               3 \text{ if } dim(V)=4\\
               2 \text{ if } dim(V)=2
            \end{cases}
\end{split}
\endeq
Assume
\beq\label{bg3.5}
dim (\pi_{1, 2} P_{\xi} (V))\le d \text{ for all } \xi.
\endeq
Taking $\xi=(0, 0)$ in \eqref{bg1.2}, clearly $\pi_{1, 2}P_{(0, 0)}(V)=P_{(0, 0)}(V)=\pi_{1, 2}(V)$. Hence by \eqref{bg3.5},
\beq\label{bg3.6}
dim(\pi_{1, 2}(V))\le d \text{ and } dim(V\cap S_3)\ge dim(V)-d.
\endeq
Recall that
\beq
\pi_{1, 2} P_{\xi}\Big|_{S_1\oplus S_2}=1_{S_1\oplus S_2} \text{ for all } \xi.
\endeq
We need the following additional lemmas.

\begin{lem}\label{bglemma2}
Fix $f\in S_1\oplus S_2$ with $f\neq 0$. Then
\beq
dim[\pi_{1, 2}P_{\xi} (S_3)+[f]]=4 \text{ for $\xi$ almost surely.}
\endeq
\end{lem}
\begin{proof}
In view of \eqref{bg1.4}, we need to show that for fixed $v\in \R^5\setminus \{0\}$,
\beq
\begin{pmatrix}
a & 0 & -1 & 0 & 0\\
0 & b & 0 & 0 & -1\\
b & a & 0 & -2 & 0\\
v_1 & v_2 & v_3 & v_4 & v_5
\end{pmatrix}
\endeq
has rank four for almost all $(a, b)$. The clearly amounts to the statement that $v_1+v_3 a+\frac{1}{2} v_4 b$ and $v_2+v_5 b+\frac{1}{2} v_4 a$ do not both identically vanish as functions of $a, b$.
\end{proof}

\begin{lem}\label{bglemma3}
Fix linearly independent $f, g, h$ in $S_1\oplus S_2$. Then
\beq
dim[\pi_{1, 2}P_{\xi}(S_3)+[f, g, h]]=5\text{ for $\xi$ almost surely.}
\endeq
\end{lem}
\begin{proof}
Given linearly independent vectors $v, w, z$ in $\R^5$, we need to prove that
\beq
\begin{pmatrix}
a & 0 & -1 & 0 & 0\\
0 & b & 0 & 0 & -1\\
b & a & 0 & -2 & 0\\
v_1 & v_2 & v_3 & v_4 & v_5\\
w_1 & w_2 & w_3 & w_4 & w_5\\
z_1 & z_2 & z_3 & z_4 & z_5
\end{pmatrix}
\endeq
has rank five for almost all $(a, b)$. This amounts to showing that
\beq
\begin{pmatrix}
v_1+a v_3+\frac{1}{2} b v_4 & w_1+a w_3+\frac{1}{2} b w_4 & z_1+a z_3+\frac{1}{2} b z_4\\
v_2+b v_5+\frac{1}{2} a v_4 & w_2+b w_5+\frac{1}{2} a w_4 & z_2+b z_5+\frac{1}{2} a z_4
\end{pmatrix}
\endeq
has rank two for almost all $(a, b)$. If this were not the case, then
\beq
\det \begin{pmatrix}
v_1 & w_1\\
v_2 & w_2
\end{pmatrix}
=\det \begin{pmatrix}
v_1 & z_1\\
v_2 & z_2
\end{pmatrix}
=\det \begin{pmatrix}
w_1 & z_1\\
w_2 & z_2
\end{pmatrix}
=0,
\endeq
meaning that
$\begin{pmatrix}
v_1 & w_1 & z_1\\
v_2 & w_2 & z_2
\end{pmatrix}$
has rank one. Also
\beq
\det \begin{pmatrix}
v_3 & w_3\\
v_4 & w_4
\end{pmatrix}
=\det
\begin{pmatrix}
v_4 & w_4\\
v_5 & w_5
\end{pmatrix}
=\det
\begin{pmatrix}
v_3 & w_3\\
v_5 & w_5
\end{pmatrix}
=0
\endeq
and similarly for the pairs $(v, z)$ and $(z, w)$, implying that
\beq
\begin{pmatrix}
v_3 & w_3 & z_3\\
v_4 & w_4 & z_4\\
v_5 & w_5 & z_5
\end{pmatrix}
\endeq
also has rank at most one. Hence $dim[v, w, z]\le 2$, which leads to a contradiction.
\end{proof}

We are ready to prove \eqref{bg00}.

\begin{description}
\item[Case $dim(V)=8$.]

We prove by contradiction. Suppose $dim(\pi_{1, 2} P_{\xi}(V))\le 4$. It follows from \eqref{bg3.6} that $dim(\pi_{1, 2}(V))\le 4$, $dim(V\cap S_3)\ge 4$, hence $S_3\subset V$, $dim(\pi_{1, 2}(V))=4$,
\beq\label{bg5.1}
V=\pi_{1, 2}(V)\oplus S_3,
\endeq
\beq\label{bg5.2}
\pi_{1, 2}P_{\xi} (V)=\pi_{1, 2}(V)+\pi_{1, 2}P_{\xi}(S_3).
\endeq
It follows from Lemma \ref{bglemma3} that $dim(\ref{bg5.2})=5$ for $\xi$ almost surely. Contradiction.  \\

\item[Case $dim(V)=6$.]

We prove by contradiction. Assuming the contrary, it follows that there exists $d$ such that $3\ge d\ge dim(\pi_{1, 2}(V))$, $dim(V\cap S_3)\ge 6-d\ge 3$ by \eqref{bg3.6}.

{\bf Case 1: $dim(\pi_{1, 2}(V))=3$.}

Then $dim(V\cap S_3)=3$ and $V$ ia a co-dimension one subspace of $\pi_{1, 2}(V)\oplus S_3$. Hence $\pi_{1, 2}P_{\xi}(V)$ is a subspace of $\pi_{1, 2} P_{\xi}(S_3)+\pi_{1, 2}(V)$ of co-dimension at most one. By Lemma \ref{bglemma3}, the latter space is of dimension five almost surely, implying $dim(\pi_{1, 2}P_{\xi}(V))\ge 4$ almost surely. Contradiction.

{\bf Case 2: $dim(\pi_{1, 2}(V))<3$.}

Then necessarily $dim(\pi_{1, 2}(V))=2$, $S_3\subset V$ and
$$V=\pi_{1, 2}(V)\oplus S_3.$$
Then
$$\pi_{1, 2} P_{\xi}(V)=\pi_{1, 2}P_{\xi}(S_3)+\pi_{1, 2}(V)$$
has dimension four almost surely by Lemma \ref{bglemma2}. Contradiction.\\

\item[Case $dim(V)=4$.] We prove by contradiction. If $dim(\pi_{1, 2}P_{\xi}(V))\le 2$, then $dim(\pi_{1, 2}(V))\le 2$ and $dim(V\cap S_3)\ge 2$.

{\bf Case 1: $dim(\pi_{1, 2}(V))\ge 1$.}

By Lemma \ref{bglemma1}, $dim(\pi_{1, 2}P_{\xi}(V\cap S_3))\ge 2$ almost surely and since $dim(\pi_{1, 2}P_{\xi}(S_3))=3$, it follows that $\pi_{1, 2}P_{\xi}(V)$ is of co-dimension at most one in
$$\pi_{1, 2}P_{\xi}(V)+\pi_{1, 2}P_{\xi}(S_3)=\pi_{1, 2}P_{\xi}(\pi_{1, 2}(V)\oplus S_3)=\pi_{1, 2}(V)+\pi_{1, 2}P_{\xi}(S_3).$$ By Lemma \ref{bglemma2}, this space is of dimension at least four almost surely. Hence $$dim(\pi_{1, 2}P_{\xi}(V))\ge 3$$ almost surely. Contradiction.

{\bf Case 2: $dim(\pi_{1, 2}(V))= 0$.}

Hence $V=S_3$. $\pi_{1, 2}P_{\xi} (V)=\pi_{1, 2}P_{\xi} (S_3)$ is of dimension three almost surely. Contradiction. \\

\item[Case $dim(V)=2$.] If $dim(\pi_{1, 2}P_{\xi}(V))\le 1$, then $dim(\pi_{1, 2}(V))\le 1$ and $dim(V\cap S_3)\ge 1$.

{\bf Case 1: $dim(\pi_{1, 2}(V))= 1$.}

Since $dim(V\cap S_3)=1$, it follows in particular from Lemma \ref{bglemma1} that $$dim(\pi_{1, 2}P_{\xi}(V\cap S_3))=1$$ almost surely. Hence $\pi_{1, 2}P_{\xi} (V\cap S_3)$ is of co-dimension two in $\pi_{1, 2}P_{\xi}(S_3)$ and $\pi_{1, 2}P_{\xi}(V)$ is of co-dimension at most two in the space
$$\pi_{1, 2}P_{\xi}(V)+\pi_{1, 2}P_{\xi}(S_3)=\pi_{1, 2}(V)+\pi_{1, 2}P_{\xi}(S_3).$$
By Lemma \ref{bglemma2}, this space is of dimension four almost surely, implying that $$dim(\pi_{1, 2}P_{\xi} (V))\ge 2$$ almost surely. Contradiction.

{\bf Case 2: $\pi_{1, 2}(V)=\{0\}$.}

Then $V\subset S_3$ and Lemma \ref{bglemma1} gives again a contradiction.
\end{description}

\noindent School of Mathematics, Institute for Advanced Study, Princeton, NJ 08540\\
\emph{Email address}: bourgain@math.ias.edu\\

\noindent Department of Mathematics, Indiana University, 831 East 3rd St., Bloomington IN 47405\\
\emph{Email address}: demeterc@indiana.edu\\

\noindent Department of Mathematics, Indiana University, 831 East 3rd St., Bloomington IN 47405\\
\emph{Email address}: shaoguo@iu.edu


%

\end{document}